\documentclass{article}

\newif\ifjournalversion
\journalversionfalse   

\ifjournalversion
    \newcommand{\onlypreprint}[1]{} 
    \newcommand{\onlyjournal}[1]{#1}
\else
    \newcommand{\onlypreprint}[1]{#1} 
    \newcommand{\onlyjournal}[1]{} 
\fi

\usepackage{tikz}
\usepackage{fullpage} 
\usepackage[noadjust]{cite}
\usepackage{multicol}
\usepackage{cite}
\usepackage{stmaryrd}
\usepackage{amsmath}
\usepackage{amsfonts} 
\usepackage{amsthm}
\usepackage{pgfplots}
\usepackage[hidelinks]{hyperref}

\usepackage{mathtools}
\usepackage[normalem]{ulem}

\newtheorem{theorem}{Theorem}[section]

\newtheorem{lemma}[theorem]{Lemma}

\newtheorem{claim}[theorem]{Claim}
\newtheorem{problem}{Problem}

\newtheorem{question}{Question}

\newlength{\bibitemsep}
\newlength{\bibparskip}
\let\oldthebibliography\thebibliography
\renewcommand\thebibliography[1]{%
  \oldthebibliography{#1}%
  \setlength{\parskip}{\bibitemsep}%
  \setlength{\itemsep}{\bibparskip}%
}

\usetikzlibrary { decorations.pathmorphing, decorations.pathreplacing, decorations.shapes, }

\tikzset{vtx/.style={inner sep=1.7pt, outer sep=0pt, circle, fill}}

\newcommand{\vc}[1]{\ensuremath{\vcenter{\hbox{#1}}}}
\tikzset{flag_pic/.style={scale=1}}  
\tikzset{unlabeled_vertex/.style={inner sep=1.7pt, outer sep=0pt, circle, fill}}
\tikzset{labeled_vertex/.style={inner sep=3pt, outer sep=0pt, rectangle, fill=white, draw=black}}
\tikzset{edge_color0/.style={color=black,line width=1.2pt,opacity=0.5,dashed}}
\tikzset{edge_color1/.style={color=red,  line width=1.2pt,opacity=0}} 
\tikzset{edge_color2/.style={color=blue, line width=1.2pt,opacity=1}}
\tikzset{edge_color3/.style={color=green!80!black,line width=1.5pt,densely dotted}}
\tikzset{edge_color4/.style={color=red, line width=1pt,decorate,decoration={zigzag,segment length=1mm,amplitude=0.2mm},line join=bevel}}
\tikzset{edge_color5/.style={color=magenta,  line width=1.4pt,dash pattern = {on 4pt off 2pt on 1pt off 2pt}}}
\tikzset{edge_color6/.style={color=gray, line width=1.4pt,densely dashed}}
\tikzset{edge_color7/.style={color=cyan, line width=1.4pt,dash pattern = {on 2pt off 3pt}}}
\tikzset{edge_color8/.style={color=gray, line width=1.2pt}}
\tikzset{edge_color9/.style={color=gray, dotted, line width=1.2pt}}
\tikzset{edge_color10/.style={color=gray, dashed, line width=1.2pt}}
\tikzset{edge_color11/.style={color=pink, dashed, line width=1.2pt}}
\tikzset{edge_colorroot/.style={color=red, line width=1.7pt}}
\tikzset{edge_thin/.style={color=black}}
\tikzset{edge_hidden/.style={color=black,dotted,opacity=0}}
\tikzset{vertex_color0/.style={inner sep=1.7pt, outer sep=0pt, draw, circle, fill=black}}
\tikzset{vertex_color1/.style={inner sep=1.7pt, outer sep=0pt, draw, circle, fill=red}}
\tikzset{vertex_color2/.style={inner sep=1.7pt, outer sep=0pt, draw, circle, fill=blue}}
\tikzset{vertex_color3/.style={inner sep=1.7pt, outer sep=0pt, draw, circle, fill=green!80!black}}
\tikzset{vertex_color4/.style={inner sep=1.7pt, outer sep=0pt, draw, circle, fill=pink}}
\tikzset{vertex_color5/.style={inner sep=1.7pt, outer sep=0pt, draw, circle, fill=gray,label=below:{$5$}}}
\tikzset{vertex_color6/.style={inner sep=1.7pt, outer sep=0pt, draw, circle, fill=gray,label=below:{$6$}}}
\tikzset{vertex_color7/.style={inner sep=1.7pt, outer sep=0pt, draw, circle, fill=gray,label=below:{$7$}}}
\tikzset{vertex_color8/.style={inner sep=1.7pt, outer sep=0pt, draw, circle, fill=gray,label=below:{$8$}}}
\tikzset{vertex_color9/.style={inner sep=1.7pt, outer sep=0pt, draw, circle, fill=gray,label=below:{$9$}}}
\tikzset{vertex_color10/.style={inner sep=1.7pt, outer sep=0pt, draw, circle, fill=gray,label=below:{$10$}}}
\tikzset{vertex_color11/.style={inner sep=1.7pt, outer sep=0pt, draw, circle, fill=gray,label=below:{$11$}}}
\tikzset{vertex_color12/.style={inner sep=1.7pt, outer sep=0pt, draw, circle, fill=gray,label=below:{$12$}}}
\tikzset{vertex_color13/.style={inner sep=1.7pt, outer sep=0pt, draw, circle, fill=gray,label=below:{$13$}}}
\tikzset{vertex_color14/.style={inner sep=1.7pt, outer sep=0pt, draw, circle, fill=gray,label=below:{$14$}}}
\tikzset{labeled_vertex_color0/.style={inner sep=2.2pt, outer sep=0pt, draw, rectangle, fill=black}}
\tikzset{labeled_vertex_color1/.style={inner sep=2.2pt, outer sep=0pt, draw, rectangle, fill=red}}
\tikzset{labeled_vertex_color2/.style={inner sep=2.2pt, outer sep=0pt, draw, rectangle, fill=blue}}
\tikzset{labeled_vertex_color3/.style={inner sep=2.2pt, outer sep=0pt, draw, rectangle, fill=green}}
\tikzset{labeled_vertex_color4/.style={inner sep=2.2pt, outer sep=0pt, draw, rectangle, fill=pink}}
\tikzset{labeled_vertex_color5/.style={inner sep=2.2pt, outer sep=0pt, draw, rectangle, fill=gray,label=below:{$5$}}}
\tikzset{labeled_vertex_color6/.style={inner sep=2.2pt, outer sep=0pt, draw, rectangle, fill=gray,label=below:{$6$}}}
\tikzset{labeled_vertex_color7/.style={inner sep=2.2pt, outer sep=0pt, draw, rectangle, fill=gray,label=below:{$7$}}}
\tikzset{labeled_vertex_color8/.style={inner sep=2.2pt, outer sep=0pt, draw, rectangle, fill=gray,label=below:{$8$}}}
\tikzset{labeled_vertex_color9/.style={inner sep=2.2pt, outer sep=0pt, draw, rectangle, fill=gray,label=below:{$9$}}}
\tikzset{text_color0/.style={color=black}}
\tikzset{text_color1/.style={color=red}}
\tikzset{text_color2/.style={color=blue}}
\tikzset{text_color3/.style={color=green!70!black}}
\tikzset{text_color4/.style={color=orange}}
\tikzset{text_color5/.style={color=gray}}

\def\outercycle#1#2{
\pgfmathtruncatemacro{\plusone}{#1+1}
\pgfmathtruncatemacro{\zeroshift}{270 - (#2-1)*360/#1/2 }
    \draw  \foreach \x in {0,1,...,#1}{(\zeroshift+\x*360/#1:1) coordinate(x\x)};
}

\def\labelvertex#1{\pgfmathtruncatemacro{\vertexlabel}{#1+1 } \draw (x#1) node{\color{black}\tiny\vertexlabel}; }

\usepackage{ifthen}

\tikzset{vertex_u/.style={unlabeled_vertex}}
\tikzset{vertex_l/.style={labeled_vertex}}

\newcommand{\Fuu}[1]{
\,\vc{\begin{tikzpicture}[scale=0.3]\outercycle{2}{1}
\draw[edge_color#1] (x0)--(x1);
\draw (x0) node[unlabeled_vertex]{};\draw (x1) node[unlabeled_vertex]{};
\end{tikzpicture}}
\,
}

\newcommand{\Fll}[1]{
\,\vc{\begin{tikzpicture}[scale=0.3]\outercycle{2}{1}
\draw[edge_color#1] (x0)--(x1);
\draw (x0) node[labeled_vertex]{};\draw (x1) node[labeled_vertex]{};
\labelvertex0
\labelvertex1
\end{tikzpicture}}
\,
}

\newcommand{\Fllu}[3]{
\vc{\begin{tikzpicture}[scale=0.4]\outercycle{3}{2}
\draw[edge_color#1] (x0)--(x1);\draw[edge_color#2] (x0)--(x2);  \draw[edge_color#3] (x1)--(x2);
\draw (x0) node[labeled_vertex]{};\draw (x1) node[labeled_vertex]{};\draw (x2) node[unlabeled_vertex]{};
\labelvertex0
\labelvertex1
\end{tikzpicture}}}

\newcommand{\FfourEdges}[6]{
\draw[edge_color#1] (x0)--(x1);\draw[edge_color#2] (x0)--(x2);\draw[edge_color#3] (x0)--(x3);  \draw[edge_color#4] (x1)--(x2);\draw[edge_color#5] (x1)--(x3);  \draw[edge_color#6] (x2)--(x3);
}
\newcommand{\Ffour}[5]{
\vc{\begin{tikzpicture}[scale=0.4]\outercycle{4}{2}
\FfourEdges#5
\draw (x0) node[vertex_#1]{};\draw (x1) node[vertex_#2]{};\draw (x2) node[vertex_#3]{};\draw (x3) node[vertex_#4]{};
\ifthenelse{\equal{#1}{l}}{\labelvertex{0}}{}%
\ifthenelse{\equal{#2}{l}}{\labelvertex{1}}{}%
\ifthenelse{\equal{#3}{l}}{\labelvertex{2}}{}%
\ifthenelse{\equal{#4}{l}}{\labelvertex{3}}{}%
\end{tikzpicture}}
}

\newcommand{\Fuuuu}[6]{\Ffour{u}{u}{u}{u}{#1#2#3#4#5#6}}

\newcommand{\Flluu}[6]{\Ffour{l}{l}{u}{u}{#1#2#3#4#5#6}}

\usepackage{listofitems}
\newcommand{\Fnv}[2]{ 
\ifnum#2<#1 \draw (x#2) node[vertex_l]{}; \labelvertex{#2}  
\else  \draw (x#2) node[vertex_u]{}; \fi 
}
\newcommand{\Fne}[3]{
\draw[edge_color#3] (x#1)--(x#2); 
}
\newcounter{Fneid} 
\newcommand{\Fn}[3]{
\ifnum#1=1
  \vc{\begin{tikzpicture}[scale=0.4]\outercycle{1}{2}
  \Fnv{#2}{0}
  \end{tikzpicture}}
\else
\setsepchar{ }
\readlist\elabel{#3}
\pgfmathtruncatemacro{\vertexloop}{#1-1}
\pgfmathtruncatemacro{\vertexloopi}{#1-2}
\pgfmathtruncatemacro{\expectededges}{#1*(#1-1)/2}
\ifnum\elabellen=\expectededges%
\def\cycleshift{2}
\ifnum#1=2\def\cycleshift{1}\fi
\ifnum#1=3\ifnum#2=1\def\cycleshift{1}\fi\fi
\vc{\begin{tikzpicture}[scale=0.4]
          \outercycle{#1}{\cycleshift}
          \setcounter{Fneid}{1}         
          \foreach\i in {0,...,\vertexloopi}{   
          \pgfmathtruncatemacro{\jfrom}{\i+1}
          \foreach\j in {\jfrom,...,\vertexloop}{
          \edef\eID{\arabic{Fneid}}
          \edef\eij{\elabel[\eID]}         
            \Fne\i\j{\eij}              
	    \stepcounter{Fneid}
          }}
          \foreach\i in {0,...,\vertexloop}{\Fnv{#2}{\i}  
          }
          \end{tikzpicture}}%
\else
   #1 vertices need \expectededges{} edges but got \elabellen edges.
\fi
\fi
}

\title{Density of rainbow triangles and properly colored $K_4$'s}
\author{
J\'ozsef Balogh\thanks{Department of Mathematics, University of Illinois Urbana-Champaign, Urbana, IL, USA, and Extremal Combinatorics and Probability Group (ECOPRO), Institute for Basic Science (IBS), Daejeon, South Korea. Email: \texttt{jobal@illinois.edu}. Partially supported by NSF grants RTG DMS-1937241, FRG DMS-2152488, (UIUC  Campus Research Board Award RB26026), the Simons Fellowship, Simons Collaboration grant, and the Institute for Basic Science (IBS-R029-C4).
} 
\and
Peter Bradshaw \thanks{Department of Mathematics, University of Illinois Urbana-Champaign, Urbana, IL, USA. Email: \texttt{pb38@illinois.edu}. Partially supported by NSF grant RTG DMS-1937241 and an AMS-Simons Travel Grant [SFI-MPS-TSM-00013107, JB].}
\and 
Ramon I. Garcia\thanks{Department of Mathematics, University of Illinois Urbana-Champaign, Urbana, IL, USA. Email: \texttt{rig2@illinois.edu}. Partially supported by NSF grant RTG DMS-1937241 and the R.H. Schark Fellowship.
}
\and
Bernard Lidick\'{y}\thanks{Department of Mathematics, Iowa State University, Ames, IA. E-mail: \texttt{lidicky@iastate.edu}. Research of this author is supported in part by NSF grant FRG DMS-2152490, Simons Collaboration grant and a Scott Hanna professorship.}
}
\date{\today}

\begin{document}

\maketitle

\begin{abstract}
We establish a sharp upper bound on the number of properly $3$-edge-colored $K_4$'s in graphs with $R$ red, $G$ green and $B$ blue edges. We give a computer-free flag-algebra proof of this bound, and we also convert our proof into a classical counting proof and an entropy proof.

Additionally, for every $k\ge 4$, for a fixed rainbow coloring $F$ of a complete graph $K_k$, we give a sharp upper bound on the number copies of $F$ in a $\binom{k}{2}$-edge-colored graph. Our proof of this result relies on a new flag-algebra version of H\"older's inequality. 

We also give a computer-free flag-algebra proof of the fact that a graph with $R$ red, $G$ green, and $B$ blue edges has at most $\sqrt{2 RGB}$ rainbow triangles, which was originally proven by
T.-W. Chao and H.-H. H. Yu using the entropy method. We also provide an even shorter entropy proof of their result.
\end{abstract}

\section{Introduction}

The following is a classical question in graph theory: 
Given graphs $H_1$ and $H_2$,
if $G$ has a fixed number of copies of $H_1$, what is the maximum number of copies of $H_2$ that $G$ can have?
The earliest instance of this question is Mantel's theorem \cite{Mantel}, 
which 
states that a triangle-free graph
on $n$ vertices contains at most $\frac 14n^2$ edges.
Tur\'an's theorem \cite{Turan} from 1941
similarly determines  the maximum number of edges in a $K_r$-free graph on $n$ vertices.
Zykov's theorem \cite{Zykov} from 1949 
generalizes Tur\'an's theorem, 
showing that Tur\'an's extremal construction also maximizes the number of copies of $K_s$ 
in a $K_r$-free graph for each $s < r$.

In 1972, a problem of Erd\H os and S\'os appeared in~\cite{Erdos1972} asking for
the maximum number of rainbow triangles in a $3$-edge-colored graph on $n$ vertices. They conjectured that the extremal construction is given by the iterated blowup of a properly colored $K_4$, which contains $(\frac 1{15} + o(1)) n^3$ rainbow triangles (see  Figure~\ref{fig:iter}(b)).
\begin{figure}
\def\e{1.5}
\def\x{13}
\def\k{
\draw
(-1.5,1.5) node[circle,draw, inner sep=\x](x1) {}
(1.5,1.5) node[circle,draw, inner sep=\x](x2) {}
(1.5,-1.5) node[circle,draw, inner sep=\x](x3) {}
(-1.5,-1.5) node[circle,draw,, inner sep=\x](x4) {}
;
\draw[green!80!black,line width=\x pt](x1)--(x2)(x3)--(x4);
\draw[red,line width= \x pt](x1)--(x3)(x2)--(x4);
\draw[blue,line width= \x pt](x1)--(x4)(x2)--(x3);
}
\begin{center}
\begin{tikzpicture}
\def\x{13}
\begin{scope}[scale=0.75]
\k
\end{scope}
\draw (0,-2) node{(a)};
\end{tikzpicture}
\hskip 3em
\begin{tikzpicture}
\def\x{13}
\begin{scope}[scale=0.75]
\k
\end{scope}
\foreach \a in {-1.5*0.75,1.5*0.75}{
\foreach \b in {-1.5*0.75,1.5*0.75}{
\def\x{4}
\begin{scope}[xshift=\a cm,yshift=\b cm,scale=0.2] \k \end{scope}
\foreach \c in {-0.3,0.3}{
\foreach \d in {-0.3,0.3}{
\begin{scope}[xshift= \a cm +\c cm,yshift=\b cm + \d cm,scale=0.06]\def\x{1.2} \k \end{scope}
}}
}
}
\draw (0,-2) node{(b)};
\end{tikzpicture}
\end{center}
\caption{(a) A blowup of a properly $3$-edge-colored $K_4$. (b) An iterated blowup of a properly $3$-edge-colored~$K_4$.}\label{fig:iter}
\end{figure}
This conjecture was solved in~\cite{BaloghRainbow2017}
for  $n$ that are powers of $4$ or sufficiently large.
Their proof used flag algebras and a stability method.


Recently, Chao and Hans Yu~\cite{ChaoEntropy2024,ChaoHungEntropic} used the method of entropy to establish another sharp upper bound on the number rainbow triangles. Unlike the result above, which is expressed in terms of the number of vertices, their upper bound  is
a function of the number of red, green and blue edges.

\begin{theorem}[Chao and Hans Yu~\cite{ChaoEntropy2024}]\label{triangleent}
    Let $\Gamma=(V,E)$ be a simple graph, and color the edges of $\Gamma$  with red, green, and blue. Denote $R$, $G$, $B$  the number of red, green, blue edges, respectively,  and $T$  the number of rainbow triangles in $\Gamma$. Then, $T^2\leq 2RGB$. 
\end{theorem}

The bound in Theorem \ref{triangleent} is tight when $\Gamma$ is a balanced blowup of a properly $3$-edge-colored $K_4$, shown in Figure~\ref{fig:iter}(a).

Theorem \ref{triangleent}
is motivated by  the \emph{joint} problem,
which asks to determine the maximum number of joints in $\mathbb{R}^d$ determined by $N$ lines, where a {\it joint} is a point $P$ with a $d$-tuple of lines intersecting at
$P$ that spans the entire space $\mathbb{R}^d$.
 The joint problem was proposed 
 in~\cite{Chazelle1992Counting}. There, a construction of $N$ lines with
many joints was given: Choose $k$ hyperplanes in $\mathbb{R}^d$ in general position.  Then, the intersection of every $(d-1)$-tuple of the hyperplanes is a line,  giving 
$N=\binom{k}{d-1}$ lines,
and every $d$-tuple of planes provides a joint, 
generating $\binom{k}{d}$ joints. Guth~\cite[Section 2.5]{Guth2016Polynomial} conjectured that this construction is optimal, and Guth's conjecture was verified asymptotically by Yu and Zhao~\cite{Yu2023Joints}, and exactly  by Chao and Yu~\cite{chao2023tight}.

Motivated by this construction, Yu and Zhao~\cite{Yu2023Joints} defined {\it  generically
induced configurations} as follows:
Let $\mathcal H$ be a set of hyperplanes in $\mathbb R^d$,
and let $L$ be a subset of the $(d-1)$-intersections of elements of $\mathcal H$.
In this setting, every joint given by the line set $L$ is an intersection point of $d$ hyperplanes from $\mathcal H$.
A generically induced configuration $(\mathcal H,L)$
has a natural representation as a $(d-1)$-uniform hypergraph,
where each hyperplane of $\mathcal H$ corresponds to a vertex, and each line in $L$ corresponds to a hyperedge.
In the hypergraph representation, a joint corresponds to a complete $(d-1)$-uniform hypergraph on $d$ vertices.
In particular, in $\mathbb R^3$,
a joint corresponds to a triangle in a graph.

The following \emph{multijoint} problem was introduced by Zhang~\cite{Zhang2020Multijoints}, who attributes the problem to Carbery:
Let $L_1,L_2,L_3$ be three families of lines in $\mathbb{R}^3$; 
what is the maximum number of joints consisting of one line from each set $L_i$?
When $L_1$, $L_2$, and $L_3$
are obtained by partitioning the line set of a generically induced configuration $(\mathcal H, L)$ into distinct color classes, this question asks for the maximum number of rainbow triangles in the graph representation of $(\mathcal H,L)$.

\section{New Results}

Motivated by Theorem \ref{triangleent} and its applications to the multijoint problem,
we continue the study of densities of multicolored subgraphs of edge colored graphs, with our bounds expressed as a function of the number of edges of each color.
We first investigate a problem analogous to that of Theorem \ref{triangleent}, where instead of the number of rainbow triangles, we are interested in the number of properly $3$-edge-colored $K_4$'s.

\begin{theorem}
\label{thm:K4-counting}
    Let $\Gamma$ be a graph with $R$ red edges, $G$ green edges, and $B$ blue edges, and suppose that $\Gamma$ has $K$ properly colored  $K_4$'s. Then, $K \leq \frac 14 (RGB)^{2/3}$.
    Furthermore, if $K = \frac 14 (RGB)^{2/3} > 0$,
    then $\Gamma$ is obtained from a balanced blowup of a properly colored $K_4$, 
     by possibly adding a set of isolated vertices.
\end{theorem}

Note that we actually prove a slightly stronger result:
\begin{equation}
K\ \le \ \frac{1}{4}\min\{RG, GB, RB\}.
    \end{equation}

Given a $3$-edge-colored graph corresponding to a generically induced configuration $(\mathcal H,L)$
whose line set is partitioned into three color classes $L_1,L_2,L_3$,
a properly colored $K_4$ subgraph corresponds to a tetrahedron
given by four planes of $\mathcal H$
in which opposite edges have the same color
and each vertex is a multijoint.
Therefore, Theorem \ref{thm:K4-counting}
gives an exact solution for the maximum number of such properly colored tetrahedra in terms of $|L_1| |L_2||L_3|$.
While this problem is somewhat contrived,
according to~\cite{Yu2023Joints},  a problem is ``a rare instance in incidence geometry where the sharp constant is determined."

While Theorem \ref{triangleent} gives a sharp upper bound for the number of multijoints in a $3$-colored generically induced configuration, it is unknown if there exists a $3$-colored joint configuration that is not generically induced whose number of multijoints exceeds the bound in Theorem \ref{triangleent}. 
On the other hand, for the properly $3$-colored tetrahedron problem, a simple construction shows that the bound in Theorem \ref{thm:K4-counting}
for the number of properly $3$-colored tetrahedron in a generically induced $3$-colored multijoint configuration does not hold for configurations that are not generically induced. Indeed, Theorem \ref{thm:K4-counting} implies that a generically induced multijoint configuration with  $L_1$ red lines, $L_2$ green lines, and $L_3$ blue lines has at most $\frac 14 (L_1 L_2 L_3)^{2/3}$ properly colored tetrahedra.
In particular, a generically induced configuration with $N$ lines gives at most $\frac{1}{36} N^2$ properly colored tetrahedra.
However, the following is a $3$-colored joint configuration that is not generically induced and that has $N$ lines and 
$(\frac 1{16} + o(1)) N^2$
properly  colored tetrahedra.
Let $L_1$ and $L_2$ be skew lines in $\mathbb R^3$.
Designate two disjoint sets $A_1$ and $A_2$, each of $n$ points, on $L_1$.
Similarly, designate two disjoint sets $B_1$ and $B_2$, each of $n$ points, on $L_2$.
For each pair $p \in A_i$ and $q \in B_j$, add a green line containing $p$ and $q$ if and only if $i=j$; otherwise, add a blue line containing $p$ and $q$.
Then, every $C_4$ that is properly colored with green and blue gives a properly $3$-colored tetrahedron. Thus, we have $N = 4n^2 + 2$ lines and $n^4 = (\frac 1{16} + o(1)) N^2$ properly colored $K_4$'s.

Note that the non-multicolor version of the
tetrahedron
problem also seems to be new, but solving it
likely
requires different methods.

\begin{problem}
\label{prob:tet}
Given $N$ lines in $\mathbb{R}^3$, what is the maximum number tetrahedra 
that they determine?
\end{problem}
 It is natural to think that the answer to Problem~\ref{prob:tet}
  comes from planes;
 that is, $k$ planes generate $N=\binom{k}{2}$ lines, which generate $\binom{k}{4} = (\frac 16 + o(1))N^2$ tetrahedra.
However, 
similarly to the multicolored setting,
there is a better construction.
Let $L_1$ and $L_2$ be two skew lines in $\mathbb R^3$.
Select $n$ points on each line,
and join 
these two sets of $n$ points with a $K_{n,n}$. Every $C_4$ in this $K_{n,n}$ gives a tetrahedron. There are $N = n^2 + 2$ lines and $\binom n2^2 = (\frac 14 + o(1)) N^2$ tetrahedra.
We can see that this second construction is not generically induced, as a point may belong to $n+1$ lines.
 
 It is natural to ask if there is a rainbow version of Theorem \ref{thm:K4-counting}:

\begin{question}\label{Q2}
Let $\Gamma$ be a graph with edges colored by colors $C=\{1,\ldots,6\}$. Denote by $C_i$  the number of edges colored by color $i$.
What is  the number of rainbow copies of $K_4$ in $\Gamma$? 
\end{question}

We conjectured in an earlier version of our paper, that the answer should be 
$H \leq \sqrt[3]{\prod_iC_i}$, as it 
 would be tight for a blow-up of a fixed rainbow coloring of $K_4$, as well as blow-ups of some graphs on six vertices.
 However, as it was discovered by Bowen Li, one can find better constructions. Our current best construction is depicted in Figure~\ref{fig:Bowen}.
The construction is obtained from 
a 6-edge-coloring of $K_8$, where the colors have multiplicities $4,4,4,4,6,6,$ and it contains $24$ rainbow $K_4$'s.
We blow this graph up, by replacing each vertex with an independent set of size $t$, giving 
color multiplicities of 
$4t^2,4t^2,4t^2,4t^2,6t^2,6t^2$.
Furthermore, the number of rainbow $K_4$ copies is $24t^4$, which is more than 
$\sqrt[3]{(4t^2)^4 (6t^2)^2} \approx 20.96 t^4$. 
The construction was found by AlphaEvolve~\cite{AlphaEvolve}.

\begin{figure}
\def\e{1.5}
\def\x{2}
\def\w{8}
\def\t{-2}
\def\k{
\draw
(-5,3+\t) node[circle,draw, inner sep=\x](a1) {}
(0,4+\t) node[circle,draw, inner sep=\x](a2) {}
(5,3+\t) node[circle,draw, inner sep=\x](a3) {}
(-6,-2) node[circle,draw, inner sep=\x](b1) {}
(-2.5,-4.5) node[circle,draw, inner sep=\x](b2) {}
(2.5,-4.5) node[circle,draw, inner sep=\x](b3) {}
(6,-2) node[circle,draw, inner sep=\x](b4) {};

\draw[black,line width=\w pt](a3)--(a2);
\draw[black,line width=\w pt](a3)--(a1);
\draw[black,line width=\w pt](a1)--(a2);

\draw[purple,line width=\w pt](b3)--(b2);
\draw[purple,line width=\w pt](b3)--(b1);
\draw[purple,line width=\w pt](b1)--(b2);
\draw[purple,line width=\w pt](b1)--(b4);
\draw[purple,line width=\w pt](b2)--(b4);
\draw[purple,line width=\w pt](b3)--(b4);

\draw[red,line width= \w pt](a1)--(b1);
\draw[blue,line width= \w pt](a1)--(b2);
\draw[yellow!90!black,line width= \w pt](a1)--(b3);
\draw[green!80!black,line width=\w pt](a1)--(b4);

\draw[red,line width= \w pt](a2)--(b2);
\draw[blue,line width= \w pt](a2)--(b1);
\draw[yellow!90!black,line width= \w pt](a2)--(b4);
\draw[green!80!black,line width=\w pt](a2)--(b3);

\draw[red,line width= \w pt](a3)--(b3);
\draw[blue,line width= \w pt](a3)--(b4);
\draw[yellow!90!black,line width= \w pt](a3)--(b1);
\draw[green!80!black,line width=\w pt](a3)--(b2);

}
\begin{center}
\tikzset{bignode/.style={circle,draw, inner sep=10pt}}
\tikzset{mythick/.style={line width=8pt,opacity=0.9}}
\tikzset{colorA/.style={blue, mythick}}
\tikzset{colorB/.style={cyan, mythick}}
\tikzset{colorC/.style={red, mythick}}
\tikzset{colorD/.style={green!80!black, mythick}}
\tikzset{colorE/.style={gray, mythick}}
\tikzset{colorF/.style={magenta, mythick}}
\begin{tikzpicture}[scale=1]
\draw\foreach \x in {1,2,...,8}{
(-22.5+45*\x:2.8) coordinate(c\x) {} 
};
\draw
(c1)node[bignode](1){}
(c4)node[bignode](2){}
(c3)node[bignode](3){}
(c2)node[bignode](4){}
(c5)node[bignode](5){}
(c6)node[bignode](6){}
(c8)node[bignode](7){}
(c7)node[bignode](8){}
;
\draw[colorD] (2)--(3)--(5)--(6)--(2)--(5)(3)--(6);
\draw[colorA] (1)--(8)--(7)--(4)--(1)--(7)(8)--(4);
\draw[colorB] (1)--(3)(2)--(4)(8)--(5)(7)--(6);
\draw[colorC] (1)--(6)(2)--(8)(3)--(7)(4)--(5);
\draw[colorE]  (1)--(5)(2)--(7)(3)--(8)(4)--(6);
\draw[colorF] (1)--(2)(3)--(4)(5)--(7)(6)--(8);
\end{tikzpicture}
\end{center}
\caption{
A blow-up of a 6-edge-coloring of $K_8$ with 24 rainbow copies of $K_4$.}\label{fig:Bowen}
\end{figure}

We answer  a weaker version of the question, counting only a fixed rainbow coloring of $K_4$. In fact, we obtain a tight upper bound for the number copies of a fixed rainbow coloring of $K_k$ for every $k\ge 4$.


\begin{theorem}\label{thm:fixedrainbowKk}
    Let $F$ be a fixed rainbow coloring of a $k$-vertex complete graph, where $k\geq 4$.
Let $\Gamma$ be a graph with edges colored by colors $C=\{1,\ldots,\binom{k}{2}\}$. Denote by $C_i$  the number of edges colored by color $i$.
If $K$ is the number of copies of $F$ in $\Gamma$, then
$K^{k-1} \leq {\prod_iC_i}$.
\end{theorem}

A variant of Theorem~\ref{thm:fixedrainbowKk} was proved by Cairncross, Mizgerd and Mubayi
\cite[Theorem 1.3]{CAIRNCROSS2025}.
They showed that in an edge-colored $K_n$ the number of copies of a fixed rainbow coloring of a $k$-vertex complete graph $F$ is
maximized by an iterated blow-up of $F$ for all for $k\geq 11$, i.e., it is at most  $$ \binom{n}{k}\cdot \frac{k!}{k^k-k}.$$
We would like to emphasize, that in 
\cite{CAIRNCROSS2025} the host graph is a clique, while for Theorem~\ref{thm:fixedrainbowKk} the maximum is obtained via coloring a not necessarily complete graph.

While working on \cite{Yu2023Joints}, Zhao asked the first and the last authors if the multiplicative constant $2$ in Theorem~\ref{triangleent} is best possible.\footnote{He asked before it became a theorem, i.e., before~\cite{ChaoEntropy2024}.}
Using the same proof techniques as those for Theorems \ref{thm:K4-counting} and \ref{thm:fixedrainbowKk}, we develop a short flag-algebra proof for Theorem \ref{triangleent},
along with an elementary counting proof and a new entropy proof. The latter two proofs are obtained by using parts of our flag algebra argument.
We also prove uniqueness of the extremal construction.

\begin{theorem}
\label{thm:K3-stability}
Let $\Gamma$ be a graph with $R$ red edges, $G$ green edges, and $B$ blue edges.
If $\Gamma$ has $\sqrt{2 RGB} > 0$ rainbow triangles,
then 
$\Gamma$ is obtained from a blowup of a properly $3$-edge-colored $K_4$
 by possibly adding a set of isolated vertices.
\end{theorem}

 For a blowup of a properly $3$-edge-colored $K_4$, see Figure~\ref{fig:iter}(a).

Our proofs  use the method of flag algebras, introduced by Razborov~\cite{RazborovFlag2007}.
\onlypreprint{We recall the main definitions from flag algebras in the Appendix.}
One noteworthy aspect of our proofs is that they are computer-free.
Moreover, 
for most of our results,
using the plain flag algebra method with sum-of-squares would not be computationally feasible. 
An equivalent reformulation of Theorem~\ref{triangleent} in flag algebras
that is 
approachable 
using semidefinite programming 
is 
\begin{eqnarray}
\label{eqn:triangleent}
\left(\Fn30{4 2 3}\right)^2 \leq 9 \cdot \Fuu4 \times \Fuu3 \times \Fuu2.
\end{eqnarray}
Verification of this inequality requires a computation using $3$-edge-colored graphs on $6$ vertices, which  would only be possible to achieve on high-memory nodes of a supercomputer.
Moreover, the smallest case of Theorem \ref{thm:fixedrainbowKk} would require a computation using $6$-edge-colored graphs on $12$ vertices,
which is entirely infeasible.
We also 
managed to translate our proof to a counting proof by hand, which probably would not be possible in the case of a large calculation.

Inspired by the rainbow triangle problem of Erd\H os and S\'os \cite{Erdos1972},
the second author also poses the following question.
\begin{question}\label{Q1}
    What is the maximum number of  rainbow copies of a $K_4$ in
$6$-edge-colored complete graphs $K_n$?
\end{question}

The iterated blow-up of the construction depicted in Figure~\ref{fig:Bowen} has rainbow $K_4$ density 
$72/511\approx 0.14090$.
As it is a nice symmetric coloring, there is a chance that it is best possible.
A flag algebra calculation on $5$ vertices gives an upper bound of $0.1552$.



In Section~\ref{sec:K4} we give three proofs of Theorem~\ref{thm:K4-counting};
one uses flag algebras, one is a counting proof and the third uses entropy. 
In Section~\ref{sec:fixedrainbowK4} we prove Theorem~\ref{thm:fixedrainbowKk}.
In Section~\ref{sec:thm11} we also give three proofs of Theorem~\ref{triangleent} using flag algebras, counting, and entropy. 
We have some concluding remarks in 
Section~\ref{sec:remarks}\onlypreprint{ and a brief introduction to flag algebras in the Appendix}.

\section{Proof of Theorem \ref{thm:K4-counting}}
\label{sec:K4}

The first goal of this section is to prove the following claim, which is the first part of Theorem \ref{thm:K4-counting}.
\begin{claim}
\label{claim:K23}
$K \leq \frac 14 \cdot \min \{RG, GB, RB\}$.
Consequently, $K \leq \frac 14 (RGB)^{2/3}$.
\end{claim}
The inequality $K \leq \frac 14 (RGB)^{2/3}$ in the claim follows by taking the geometric mean of the three upper bounds for $K$; therefore, it suffices just to prove the first part of the claim.
We give a short proof using flag algebras, which we translate to a counting proof as well as to an entropy proof.

\subsection{Flag algebra proof}
We assume the reader of this paper has  some basic backgrounds in flag algebras; we provide some introduction to flag algebras in the Appendix\onlyjournal{ of the arXiv version of the paper}, which should be sufficient to understand our proofs.

Consider a graph $\Gamma$ with $R$ red edges, $G$ green edges, and $B$ blue edges, and let $K$ be the number of properly colored $K_4$ subgraphs of $\Gamma$. First, we provide  a short  proof of 
the stronger
inequality $K \leq \frac 14 \cdot \min \{RG, GB, RB\}$
using the language of flag algebras.

\begin{lemma}
\begin{equation}
\label{eqn:gxb-flags}
\Fuu3 \times \Fuu2 \geq  \frac 13 \cdot \left ( \Fuuuu422331  + \Fuuuu422332 +  \Fuuuu422333 +  \Fuuuu422334 \right ) = 4 \cdot \left \llbracket \Fllu423 ^2 \right \rrbracket.
\end{equation}
\end{lemma}
\begin{proof}
The inequality holds, as in each of the shown graphs on four vertices, 
there are six ways to partition the four vertices into parts $A_1$ and $A_2$ of size $2$,
and for exactly two of these partitions, 
$A_1$ induces a green edge while $A_2$ induces a blue edge.
The identity holds, as 
the four $4$-vertex graphs 
comprise all $4$-vertex $3$-edge-colored graphs in which a red edge $uv$
belongs to two rainbow triangles, with $u$ opposite the green edge and $v$ opposite the blue edge in each triangle.
Furthermore, given one of these four graphs $H$, the probability that a random injection $\{1,2\} \hookrightarrow V(H)$ both maps $1$ to the unique vertex with red degree $1$ and blue degree $2$, and also maps $2$ to the unique vertex with red degree $1$ and green degree $2$, is $\frac 1{12}$.
\end{proof}

\begin{lemma}\label{lemma:FA4}
\begin{align}
    \label{eq:K4flag}
\Fuuuu423324 \leq \frac 32 \cdot \min\left\{ \Fuu4 \times \Fuu3,\quad   \Fuu3 \times \Fuu2,\quad   \Fuu4 \times \Fuu2  \right\}.
\end{align}    
\end{lemma}

\begin{proof}
We use Razborov's Cauchy-Schwarz inequality for flags
\cite[Theorem 3.14]{RazborovFlag2007},
which states that two flags $F,G$ of a type $\sigma$ satisfy $\llbracket F \times G
\rrbracket
\leq \sqrt{\llbracket F^2 \rrbracket \llbracket G^2 \rrbracket } $.
Combining this Cauchy-Schwarz inequality with the inequality~\eqref{eqn:gxb-flags}, we obtain
\begin{equation}
\Fuuuu423324 \ = \ 6 \cdot \left\llbracket \frac 12  \Flluu423324 \right\rrbracket 
\  \leq \ 6 \cdot \left\llbracket \Fllu432 \times \Fllu423 \right\rrbracket 
\  \leq \  6\cdot\sqrt{\left\llbracket \left(\Fllu432\right)^2 \right\rrbracket \cdot \left\llbracket \left(\Fllu423\right)^2 \right\rrbracket}.
\end{equation}
Since 
\[
\left\llbracket \left(\Fllu423\right)^2 \right\rrbracket = \left\llbracket \left(\Fllu432\right)^2 \right\rrbracket,
\]
we obtain
\[
\Fuuuu423324 \leq 6 \cdot \left\llbracket \Fllu423^2 \right\rrbracket \leq \frac 32\cdot  \Fuu2 \times \Fuu3,
\]
with the last inequality following from~\eqref{eqn:gxb-flags}.
By repeating the argument with a labeled green edge type and a labeled blue edge type,
\[
\Fuuuu423324\leq \frac 32 \cdot \Fuu2 \times \Fuu4 \quad \quad\quad\text{ and } \quad\quad\quad\Fuuuu423324\leq \frac 32 \cdot \Fuu4 \times \Fuu3.
\]
The three inequalities above imply \eqref{eq:K4flag}.
\end{proof}

We note that Lemma~\ref{lemma:FA4} 
can be strengthened by using a sum-of-squares argument, which is  typical for applications of flag algebras:
\[
\Fuuuu423324 \leq \Fuuuu423324 + 3 \cdot \left\llbracket\left( \Fn32{4 2 3} - \Fn32{4 3 2}\right)^2\right\rrbracket   \leq  \frac32 \cdot\Fn202 \times \Fn203.
\]
The last inequality above can be checked using the equations and inequalities in the proof of Lemma \ref{lemma:FA4}:

\begin{align*}
\Fuuuu423324 &\leq \Fuuuu423324 + 3 \cdot \left\llbracket\left( \Fn32{4 2 3} - \Fn32{4 3 2}\right)^2\right\rrbracket  
= \Fuuuu423324 + 3 \cdot \left\llbracket
\Fn32{4 2 3}^2 -2 \cdot\Fn32{4 2 3} \times\Fn32{4 3 2} + \Fn32{4 3 2}^2
\right\rrbracket \\
&= \Fuuuu423324 + 3 \cdot \left\llbracket
\Fn32{4 2 3}^2 \right\rrbracket -6 \cdot \left\llbracket \Fn32{4 2 3} \times \Fn32{4 3 2}  \right\rrbracket + 3\cdot \left\llbracket \Fn32{4 3 2}^2
\right\rrbracket \\
&= \underbrace{\left(\Fuuuu423324 -6 \cdot \left\llbracket \Fn32{4 2 3} \times \Fn32{4 3 2} \right\rrbracket \right)}_{\leq 0} +6 \cdot \left\llbracket
\Fn32{4 2 3}^2 \right\rrbracket \leq  6 \cdot \left\llbracket
\Fn32{4 2 3}^2 \right\rrbracket
\leq  \frac32 \cdot\Fn202 \times \Fn203.
\end{align*}

We also need a technical lemma that allows us to avoid the error terms usually associated with flag algebras.

\begin{lemma}\label{lemma:blowup}
Let $\Gamma$ be an $r$-edge-colored graph on $n$ vertices 
with $C_i$ edges of each color $i$.
Let $F$ be a $K_k$ with a fixed edge coloring, and let $H$ be the number of copies of $F$ in $\Gamma$.
Then, there exists a convergent sequence of $r$-edge-colored graphs whose limit $\phi \in \text{Hom}^+(\mathcal A, \mathbb R)$ satisfies
\[
\phi (e_i) = \frac{2C_i}{n^2} \quad \text{ for } \quad  1 \leq i \leq r, \quad\text{ and } \quad
\phi (F) = \frac{k! K}{n^k},
\]
where $e_i$ is the unlabeled edge flag on $2$ vertices of color $i$.
\end{lemma}
\begin{proof}
Let $\Gamma$ be an $r$-edge-colored graph on $n$ vertices. 
Define the sequence $(\Gamma_\ell)_{\ell\geq1}$, where $\Gamma_\ell$ is a graph obtained from $\Gamma$ by replacing each of its vertices by an independent set of $\ell$ vertices and each edge of color $c$ by a copy of $K_{\ell,\ell}$ with all edges of the $K_{\ell,\ell}$ colored $c$.
This is a convergent sequence of graphs (see \cite[Section 2.5]{LSz}), 
and its limit corresponds to  a homomorphism $\phi \in \text{Hom}^+(\mathcal{A},\mathbb{R})$.
Since each edge in $\Gamma$ is replaced by $\ell^2$ edges in $\Gamma_\ell$ of the same color, 
$\Gamma_\ell$ has $C_i\ell^2$ edges of color $i$.
Similarly,
as $F$ is a complete graph, every homomorphism $V(F) \rightarrow V(\Gamma)$ is injective; therefore,
$\Gamma_\ell$ has $F\ell^k$ copies of $F$.
A straightforward calculation shows
\[
\phi(e_i) = \lim_{\ell \rightarrow \infty} \frac{C_i \ell^2}{\binom{n \ell}{2}} = \frac{2C_i}{n^2}\quad  \text{ for } \quad 1 \leq i \leq r, \quad \text{ and } \quad \phi(F) = \lim_{\ell \rightarrow \infty} \frac{K\ell^k}{\binom{n \ell}{k}} = \frac{k!  K}{n^k}.
\]
\end{proof}
Combination of  Lemmas \ref{lemma:FA4} and \ref{lemma:blowup},
completes the proof of Claim \ref{claim:K23}.
\begin{proof}[Proof of Claim \ref{claim:K23}]
Let $\Gamma$ be a $3$-edge-colored graph on $n$ vertices. Using Lemma~\ref{lemma:blowup},
let $\phi \in \text{Hom}^+(\mathcal A, \mathbb R)$ be a limit homomorphism obtained from $\Gamma$.
By~\eqref{eq:K4flag} and the fact that $\phi$ is an algebra homomorphism,
\begin{align*}
\frac{24K}{n^4} &= \phi\left(\Fuuuu423324\right) \leq \frac 32 \cdot \min\left\{ \left ( \phi\left(\Fuu2 \times \Fuu3 \right) \right),\quad   \left ( \phi\left(\Fuu2 \times \Fuu4 \right) \right),\quad  \left ( \phi\left(\Fuu4 \times \Fuu3 \right) \right) \right\}\\ &= \frac32\cdot \min\left\{ \left(  \frac{2R}{n^2} \cdot \frac{2G}{n^2} \right),\quad  \left(  \frac{2G}{n^2} \cdot \frac{2B}{n^2} \right),\quad  \left(  \frac{2B}{n^2} \cdot \frac{2R}{n^2} \right)\right\},
\end{align*}
yielding the statement of the claim.
\end{proof}

\subsection{Counting proof}

In this section we present
 an alternative proof of Claim~\ref{claim:K23} that does not use flag algebras.
 First we need the following lemma, which closely corresponds to \eqref{eqn:gxb-flags}.

\begin{figure}[h!]
\begin{center}
    \begin{tikzpicture}
        \draw
        (0,0) node[vtx,label=left:$u$](u){}
        (1,0) node[vtx,label=right:$v$](v){}
        (0,1) node[vtx,label=left:$x$](x){}
        (1,1) node[vtx,label=right:$y$](y){}

        (3,0) node[vtx,label=left:$u$](uu){}
        (4,0) node[vtx,label=right:$v$](vv){}
        (3.5,1) node[vtx](xx){}
        (3.5,1) node[right]{$x=y$}

        (9,0) node[vtx](g1){}
        (9,1) node[vtx](g2){}
        (8,0) node[vtx](b1){}
        (8,1) node[vtx](b2){}

        (12,0) node[vtx](gg1){}
        (11.5,1) node[vtx](rr1){}
        (11,0) node[vtx](bb1){}
        ;

        \draw[edge_color4](u)--(v) (uu)--(vv);
        \draw[edge_color2](u)--(x)  (u)--(y) (b1)--node[pos=0.5,left]{$b$}(b2) (uu)--(xx) (bb1)--node[pos=0.5,left]{$b$}(rr1);
        \draw[edge_color3](v)--(x) (v)--(y) (g1)--node[pos=0.5,right]{$g$}(g2) (vv)--(xx) (gg1)--node[pos=0.5,right]{$g$}(rr1);
        \draw[edge_color0](x)--node[pos=0.5]{?}(y)  (b1)--node[pos=0.5]{?}(g1) (b1)--node[pos=0.3]{?}(g2) (g1)--node[pos=0.3]{?}(b2) (b2)--node[pos=0.5]{?}(g2) (gg1)--node[pos=0.5]{?}(bb1);
        \draw (2,-0.5) node{$S$};
        \draw (10,-0.5) node{$S'$};
    \end{tikzpicture}
\end{center}
    \caption{Members of sets $S$ and $S'$ from Lemma~\ref{lem:gxb}.}
    \label{fig:SS'}
\end{figure}

\begin{lemma}
\label{lem:gxb}
    Let $\Gamma$ be a graph whose edges are colored red, green, and blue.
    Let $S$ be the set of ordered vertex tuples $(u,v,x,y) \in V(\Gamma)^4$ for which $uv$ is a red edge,
$ux$ and $uy$ are blue edges, and $vx$ and $vy$ are green edges, as in Figure~\ref{fig:SS'}.
Let $S'$ be the set of pairs $(g,b) \in E(\Gamma)^2$ for which $g$ is an unordered green edge and $b$ is an unordered blue edge.
Then, $|S| \leq |S'|$.
\end{lemma}
\begin{proof}
To prove the lemma, we show that the function $f:S \rightarrow S'$ mapping $(u,v,x,y) \mapsto (\{u,x\}, \{v,y\})$ is injective.
Indeed, choose  an arbitrary element $(g,b)$ in the image of $f$. We show that the vertices $u,v,x,y$ can be uniquely determined from $(g,b)$.

    {\bf Case (i):} Suppose that $g$ and $b$ share an endpoint $z$. \\
    Then, the set $\{u,v,x,y\}$ contains at most three distinct vertices and induces three edges.
    Thus, $u$ is the unique vertex incident to a red edge and a blue edge, $v$ is the unique vertex incident to a red edge and a green edge,
    and $z = x = y$. 
    
    {\bf Case (ii):} Suppose that $g$ and $b$ are vertex-disjoint.\\ Then,  $u$ is the endpoint of $b$ with two blue neighbors and one red neighbor among the endpoints of $g$ and $b$. Crucially, $u$ is the unique vertex not having a green edge. Then, $x$ is the endpoint of $b$ distinct from $u$. Having identified $u$ and $x$, $v$ is the red neighbor of $u$, and $y$ is the other blue neighbor of $u$.

As each element of $S'$ in the image of $f$ has  at most one  pre-image in $S$, we have $|S| \leq |S'|$.
\end{proof}

\begin{proof}[Proof of Claim~\ref{claim:K23}]
    Let $K$ be the number of properly colored $K_4$ subgraphs of $\Gamma$.
    Let $\mathcal R$ be the set of ordered vertex pairs $uv \in V(\Gamma)^2$ for which $uv$ induces a red edge.
    For each $uv \in \mathcal R$, let $d_K(uv)$ be the number of ordered pairs $wx \in \mathcal R$ for which $uw$ and $vx$ are green, and  $ux$ and $vw$ are blue.
    Finally, for each $uv \in \mathcal R$, define $d^+(uv)$ as the number of triples $uvw \in V(G)^3$ for which $uw$ is blue and $vw$ is green,    
    and define $d^-(uv)$ as the number of triples $uvw \in V(G)^3$ for which $uw$ is green and $vw$ is blue.
    Observe that for each $uv \in \mathcal R$,  
 \begin{equation}
    \label{equ:d+d-}
         { d_K(uv) \leq d^-(uv)\cdot  d^+(uv) \quad\quad\quad\quad   \text{and}   \quad\quad\quad\quad   \sum_{uv \in \mathcal R} d^-(uv)^2 \ = \  \sum_{uv \in \mathcal R} d^+(uv)^2 }.
    \end{equation}
    Thus, by using the Cauchy-Schwarz inequality,
    \begin{equation}
    \label{eqn:d+d-}
        4K = \sum_{uv \in \mathcal R} d_K(uv) \leq \sum_{uv \in \mathcal R} d^-(uv) \cdot d^+(uv) \leq \sqrt{ \left ( \sum_{uv \in \mathcal R} d^-(uv)^2 \right ) \cdot \left ( \sum_{uv \in \mathcal R} d^+(uv)^2 \right )}= \sum_{uv \in \mathcal R} d^+(uv)^2.
    \end{equation}

Observe that for  an ordered pair $uv \in \mathcal R$,
$d^+(uv)^2$ counts the ordered pairs $xy \in V(\Gamma)^2$
for which
$ux,uy$ are blue edges
and $vx, vy$ are green edges. 
Therefore, $
\sum_{uv \in \mathcal R} d^+(uv)^2$ 
is the number of ordered vertex tuples $(u,v,x,y)$ for which $uv$ is a red edge,
$ux$ and $uy$ are blue edges, and $vx$ and $vy$ are green edges.
Combining \eqref{eqn:d+d-} and Lemma \ref{lem:gxb}, we obtain $4K\leq G B$.

By symmetry, we also have $4K \leq R  B$ and $4K \leq R  G$. Therefore, these three  upper bounds imply
    \[
        4K \leq \min\left\{RG,\text{ } G B,\text{ } R B\right\}.      
    \]
 \end{proof}

\subsection{Entropy proof}
Finally, we provide an entropy version of our proof of Claim~\ref{claim:K23}. 

\begin{proof}[Proof of Claim~\ref{claim:K23}] Our proof shares key ideas with \cite{ChaoHungEntropic}.
    We assume familiarity with the entropy of a discrete random variable $X$, defined as $H(X)=-\sum_{x}\mathbb{P}(X=x)\log_{2}(\mathbb{P}(X=x))$, where $x$ ranges over the outcomes of $X$ that occur with nonzero probability.
    
    Let $(v_1,v_2,v_3,v_4)\in V(\Gamma)^{4}$ be a tuple of vertices sampled uniformly at random from the tuples such that the edges $v_1v_2$, $v_3 v_4$ are red, $v_1 v_3$, $v_2 v_4$ are blue and $v_1 v_4$, $v_2v_3$ are green, as in Figure~\ref{fig:EntropyK4}. From the definition of $(v_1,v_2,v_3,v_4)$, it follows that $H(v_1,v_2,v_3,v_4)=\log_2(4K)$. By the chain rule,
    \begin{equation}\label{eq:entropyK4}
    \begin{split}
    H(v_1,v_2,v_3,v_4)&=H(v_1|v_2,v_3,v_4)+H(v_2,v_3,v_4)= H(v_1|v_2,v_3,v_4)+H(v_2|v_3,v_4)+H(v_3,v_4)\\
    &\leq H(v_1|v_3,v_4)+H(v_2|v_3,v_4)+H(v_3,v_4),
    \end{split}
    \end{equation}
    where the inequality follows from dropping the conditioning on $v_2$. We resample a vertex $v_{1}'$ so that the edges $v_1' v_3$ and $v_1' v_4$ are blue and green, respectively, and the vertices  $v_1$ and $v_1'$ are conditionally independent and identically distributed given $v_3,v_4$, so that $H(v_1'|v_3, v_4)=H(v_1|v_3, v_4)$. Similarly, we resample a vertex $v_2'$
    so that the edges $v_2' v_4$ and $v_2' v_3$ are blue and green, respectively, and
    the vertices  $v_2$ and $v_2'$ are conditionally independent and identically distributed given $v_3,v_4$, so that $H(v_2|v_3,v_4)=H(v_2'|v_3,v_4)$. From \eqref{eq:entropyK4} it follows that
    \begin{align*}
    H(v_1,v_2,v_3,v_4)&\leq \frac 12 \left(H(v_1|v_3,v_4)+H(v_2|v_3,v_4)+H(v_3,v_4)+H(v_1'|v_3,v_4)+H(v_2'|v_3,v_4)+H(v_3,v_4)\right)\\
    = \frac 12& \left(H(v_1,v_1'|v_3, v_4)+H(v_2,v_2'|v_3, v_4)+2H(v_3,v_4)\right)
    = \frac 12 \left(H(v_3,v_4,v_1,v_1')+H(v_3,v_4,v_2,v_2')\right).
    \end{align*}
Notice that $(v_3,v_4,v_1,v_1'),(v_4,v_3,v_2,v_2')\in S$; therefore, $H(v_3,v_4,v_1,v_1'),H(v_3,v_4,v_2,v_2')\leq \log_2(|S|)$. By Lemma~\ref{lem:gxb} we have $\log_2(|S|)\leq \log_2(G)+\log_2(B)$. We conclude
\[
4K=2^{H(v_1,v_2,v_3,v_4)}\leq 2^{\frac 12\left(H(v_3,v_4,v_1,v_1')+H(v_3,v_4,v_2,v_2')\right)}\leq G\cdot B. 
\]
\end{proof}

\begin{figure}[h!]
\begin{center}
    \begin{tikzpicture}
        \draw
        (0,0) node[vtx](a){}
        (0,0) node[left]{$v_4$}

        (2,0) node[vtx](b){}
        (2,0) node[right]{$v_3$}

        (0,2) node[vtx](c){}
        (0,2) node[left]{$v_1$}
        
        (2,2) node[vtx](d){}
        (2,2) node[right]{$v_2$};

        \draw[edge_color4](a)--(b) (c)--(d);
        \draw[edge_color3](a)--(c) (b)--(d);
        \draw[edge_color2](a)--(d) (b)--(c);
    \end{tikzpicture}\quad \quad
    \begin{tikzpicture}
        \draw
        (0,0) node[vtx,label=left:$v_4$](a){}
        (2,0) node[vtx, label=right:$v_3$](b){}
        (0,2) node[vtx, label=left:$v_1$](c){}
        (2,2) node[vtx, label=right:$v_2$](d){}
        (-1,1) node[vtx, label=left:$v_1'$](e){}
        (3,1) node[vtx, label=right:$v_2'$](f){}
        ;

        \draw[edge_color4](a)--(b) (c)--(d);
        \draw[edge_color3](a)--(c) (a)--(e) (b)--(d) (b)--(f);
        \draw[edge_color2](a)--(d) (b)--(e) (b)--(c) (a)--(f);
        \draw[edge_color0](d)--node[pos=0.6]{?}(e) (c)--node[pos=0.5]{?}(e) (d)--node[pos=0.5]{?}(f) (c)--node[pos=0.6]{?}(f) (e)--node[pos=0.4]{?}(f);
    \end{tikzpicture}
    \caption{A drawing of $\Gamma[v_1,v_2,v_3,v_4]$ and the graph obtained by adding the resampled vertices $v_1',v_2'$.}
    \label{fig:EntropyK4}
\end{center}
\end{figure}

\subsection{Proof of stability}

Finally, we prove the second part of Theorem \ref{thm:K4-counting}, which states that a $3$-edge-colored graph achieving  the upper bound in Claim \ref{claim:K23} is obtained from a balanced blowup of a properly colored $K_4$ by adding isolated vertices.
First, we need a lemma.

\begin{lemma}
\label{lem:K4char}
    Let $\Gamma$ be a graph whose edges are colored red, green, and blue. Assume  that the following conditions hold for $\Gamma$:

       (a) For each pair $e_1,e_2$ of  edges of distinct colors $c_1$ and $c_2$, there is an edge of the third color $c_3$ joining $e_1$ and $e_2$.
       
       (b) There exists  $d \geq 1$ such that every $v \in V(\Gamma)$ is incident to exactly $d$ red, $d$ green, and $d$ blue edges.
   
  \noindent  Then,
   $\Gamma$ is  a balanced blowup of a  properly colored $K_4$.
\end{lemma}
\begin{proof}
First we show that $\Gamma$ is connected.
Indeed, suppose that $V(\Gamma)$ has two disjoint subsets $V_1$ and $V_2$ such that 
$\Gamma[V_1]$ and $\Gamma[V_2]$ are  components of $\Gamma$.
By (b), $\Gamma[V_1]$ contains a green edge $g$, and $\Gamma[V_2]$ contains a blue edge $b$.
Then, by (a), a red edge joins an endpoint of $g$ with an endpoint of $b$, a contradiction.

    Now, fix a vertex $v \in V(\Gamma)$. Write $V_R \subseteq V(\Gamma)$ for the set of vertices of $\Gamma$ that are joined to $v$ by a red edge, and define $V_G$ and $V_B$ similarly.
    Define $V_0 \subseteq V(\Gamma)$ as the set of vertices in $V(\Gamma)$ that are not adjacent to $v$.
    As $v$ has exactly $d$ incident edges of each color, $|V_R| = |V_G| = |V_B| = d$.
    Furthermore, for each permutation $(c_1,c_2,c_3)$ of the three symbols $R,G,B$, (a) implies that each $u \in V_{c_1}$ and each $u' \in V_{c_2}$ are joined by an edge of a color matching the symbol $c_3$.

    Next, fix a vertex $w \in V_G$.
    As $w$ is incident to exactly $d$ edges of each color, (a) and (b) imply that the set of neighbors of $w$ in red is exactly the set $V_B$, and the set of neighbors of $w$ in blue is exactly the set $V_R$.
    Thus, as (a) also implies that every triangle is rainbow, $V_R$ and $V_B$ are independent sets.
    
    Similarly, (a) and (b) imply that for each vertex $w' \in V_B$, the set of red neighbors of $w'$ is exactly $V_G$;
    therefore, (a) implies that $V_G$ is an independent set.
    Now, let $U$ be the set of the $d$ neighbors of $w$ in green.
    As $V_G$ is an independent set, $U \subseteq V_0$.
    Furthermore,
    by (a), for each vertex $x \in U$,
    $x$ is a red neighbor of each vertex of $V_R$, and $x$ is a blue neighbor of each vertex of $V_B$, 
    and hence (a) also implies that $x$ is a green neighbor of each vertex of $V_G$.
    Thus, using (b), we obtain that $N_{\Gamma}(v')\subseteq V_{R}\cup V_{G}\cup V_{B}\cup U$ for every vertex $v'\in V_{R}\cup V_{G}\cup V_{B}\cup U$. Since $\Gamma$ is connected, it follows $U=V_0$.
    
    Therefore, $\Gamma$ is obtained from a properly colored $K_4$ by blowing up each vertex to one of the independent sets $V_R,V_G,V_B,V_0$, each of which contains exactly $d$ vertices.
\end{proof}

Now, we are ready to prove the rest of Theorem \ref{thm:K4-counting}.
We show that if $K = \frac 14 (RGB)^{2/3} > 0$, then $\Gamma$ is obtained from a balanced blowup of a properly colored $K_4$ by adding a set of isolated vertices.

If one of
the expressions $4K\leq R\cdot G$, $4K\leq R\cdot B$ and $4K\leq B\cdot G$ is a strict inequality, then it follows $(4K)^{3}<(R\cdot G\cdot B)^{2}$. Therefore, if $K=\frac 14 (R\cdot G\cdot B)^{2/3}$, then the inequality in Lemma~\ref{lem:gxb}, as well as the inequalities obtained from it by permuting colors, are tight.

We write $V' \subseteq V(\Gamma)$ for the set of vertices
with at least one incident edge.     We observe that each edge of $\Gamma$ belongs to a properly colored $K_4$. Indeed, if $e$ does not belong to a properly colored $K_4$, then $e$ could be deleted without reducing the number of properly colored $K_4$ graphs. 
    Assuming without loss of generality that $e$ is red,
    $
    T = \frac 14 (RGB)^{2/3} > \frac 14 ( (R-1)GB)^{2/3},$
    contradicting Claim \ref{claim:K23}.
    Therefore, in particular, every vertex of $V'$ is incident to an edge of each color.

 We show that if each inequality in the counting proof of Claim \ref{claim:K23} is tight, then conditions  (a) and (b) of Lemma \ref{lem:K4char} hold  for $\Gamma[V']$,
    implying that $\Gamma$ is obtained from a balanced blowup of a properly colored $K_4$ by adding a set of isolated vertices.

    First, we analyze the inequality given by Lemma \ref{lem:gxb}.
    If $|S| < |S'|$ in the lemma, 
    then the inequality 
    $
\sum_{uv \in \mathcal R} d^+(uv)^2 \leq 
G \cdot B$ in Theorem \ref{triangleent} is strict, a contradiction; therefore, $|S| = |S'|$.
Hence, for every green edge $g$ of $G$ and every blue edge $b$ of $G$, the endpoints of $g$ and $b$ form a set $\{u,v,x,y\}$ for which $uv$ is a red edge, $ux$ and $uy$ are blue edges, and $vx$ and $vy$ are green edges.
In particular, a red edge joins an endpoint of $g$ with an endpoint of $b$.
By repeating the argument with permuted colors, we observe that the condition (a) of Lemma \ref{lem:K4char} holds for $\Gamma[V']$.
Together with the fact that each $v \in V'$ is incident to an edge of each color, this implies that $\Gamma[V']$ is connected.

Next,  we claim that condition (b) of Lemma \ref{lem:K4char} holds for $\Gamma[V']$.
\begin{claim}
    There is a  $d \geq 1$ such that every $v \in V'$ is incident to exactly $d$ red, $d$ green, and $d$ blue edges.
\end{claim}
\begin{proof}
We analyze the Cauchy-Schwarz inequality 
\[
\sum_{uv \in \mathcal R} d^-(uv) \cdot d^+(uv) \leq \sqrt{ \left ( \sum_{uv \in \mathcal R} d^-(uv)^2 \right ) \cdot \left ( \sum_{uv \in \mathcal R} d^+(uv)^2 \right )}
\]
in \eqref{eqn:d+d-}.
The Cauchy-Schwarz inequality is an equality if and only if $d^+(uv) = d^-(uv)$ for each pair $uv \in \mathcal R$; therefore, we assume that $d^+(uv) = d^-(uv)$ for each $uv \in \mathcal R$.
Furthermore,
by condition (a) of Lemma \ref{lem:K4char},
$d_G(u) = d^-(uv) = d^+(uv) = d_B(u)$,
where $d_G(u)$ and $d_B(u)$ denote the number of incident green edges and incident blue edges to $u$, respectively.
By permuting colors, 
and recalling that each $u \in V'$ is incident to an edge of each color,
\begin{equation}
\label{eqn:local-balance}
    d_R(u) = d_G(u) = d_B(u) \text{ for each $u \in V'$},
\end{equation}
where $d_R(u)$ is the number of red edges incident to $u$.
Next, by condition (a) of Lemma \ref{lem:K4char}, 
for each $uv  \in \mathcal R$ and by
\eqref{eqn:local-balance},
$d_G(u) = d^-(uv) = d_B(v) = d_G(v)$.
Together with \eqref{eqn:local-balance} and the fact that $\Gamma[V']$ is connected,
condition (b)  of Lemma \ref{lem:K4char}
holds for $\Gamma[V']$.
\end{proof}
As both conditions (a) and (b) of Lemma \ref{lem:K4char} hold,
$\Gamma[V']$ is a balanced blowup of a properly colored $K_4$, completing the proof of Theorem \ref{thm:K4-counting}.

\section{Proof of Theorem~\ref{thm:fixedrainbowKk}}\label{sec:fixedrainbowK4}
We begin by establishing some notation.
Suppose that $\sigma$ is a type.
Recall that a $\sigma$-flag is formally a pair $(G,\theta)$, where $\theta$ is a model embedding of $\sigma$ in a model $G$ (in our case, an injective labeling of $|\sigma|$ vertices of a graph $G$ that induces a labeled subgraph isomorphic to $\sigma$).
We write $0$ for the type of order zero. 
Thus, a flag of type $0$ is formally a model $G$ (in our case, a graph), paired with the empty injection $\theta: \emptyset \hookrightarrow V(G)$. In other words, no vertex of $G$ is labeled. Thus, we often identify an unlabeled graph $G$ with its corresponding flag of type $0$.

Given a flag $(G,\theta)$ of some type $\sigma$, we write $(G,\theta)|_0 = G$; in other words,
we let $(\cdot) |_0$ be an operator that removes labels from a flag and hence returns a flag of type $0$.
For each type $\sigma$, we extend $(\cdot) |_0$ to act linearly on the algebra $\mathcal A^{\sigma}$ of $\sigma$-flags.

We also establish a version of H\"older's inequality for flags. 
The proof is  nearly identical to the proof of the Cauchy-Schwarz inequality from \cite[Theorem 3.14]{RazborovFlag2007}. We begin with the following form of H\"older's inequality:
\begin{lemma} [{\cite[Problem~6.31]{Folland}}]\label{lem:Holder}
    Let $S$ be a measure space, and let $p_1, \dots, p_{\ell} > 0$ be values whose reciprocals sum to $1$. For all measurable functions $f_1, \dots, f_{\ell}: S \rightarrow \mathbb R$, we have 
    \[
        \left \| \prod_{i=1}^{\ell} f_i \right \|_1 \leq \prod_{i=1}^{\ell} \|f_i\|_{p_i}.
    \]
\end{lemma}

\begin{lemma}
 If $F_1,\dots,F_{\ell}$ are flags of a common type $\sigma$, then
    \[
        \llbracket F_1 \times \ldots \times F_{\ell} \rrbracket \leq \left ( \llbracket F_1^{\ell} \rrbracket \cdot\ldots\cdot  \llbracket F_{\ell}^{\ell} \rrbracket \right )^{1/\ell}.
    \]
\end{lemma} 
\begin{proof}
Let $\mathcal A$ be the flag algebra associated with the type $0$.
To prove the lemma, we must show that for every $\phi \in \text{Hom}^+(\mathcal A, \mathbb R) $, 
\[
\phi \left (  \llbracket F_1 \times \cdots \times F_{\ell} \rrbracket \right ) ^{\ell} \leq \phi \left (\llbracket F_1^{\ell} \rrbracket \right ) \cdot\ldots\cdot  \phi \left (\llbracket F_{\ell}^{\ell} \rrbracket \right ).
\]
Fix an element $\phi \in \text{Hom}^+(\mathcal A, \mathbb R)$, 
and let $\boldsymbol{\phi}^{\boldsymbol{\sigma}}$ be the unique random homomorphism rooted at $\phi$ associated with $\sigma$ (see \cite[Theorem 5, Definition 10]{RazborovFlag2007}).
Then, we must prove that 
\[
\mathbb E[ \boldsymbol{\phi}^{\boldsymbol{\sigma}}(F_1) \cdot\ldots\cdot \boldsymbol{\phi}^{\boldsymbol{\sigma}}(F_{\ell})]^{\ell}
\leq \mathbb E [\boldsymbol{\phi}^{\boldsymbol{\sigma}}(F_1)^{\ell} ] \cdot\ldots\cdot \mathbb E[\boldsymbol{\phi}^{\boldsymbol{\sigma}}(F_{\ell} )^{\ell}] .
\]
This inequality follows immediately from Lemma \ref{lem:Holder} with $p_1 = \ldots = p_{\ell} = \ell$ by considering each random variable 
$\boldsymbol{\phi}^{\boldsymbol{\sigma}}(F_i)$ as a measurable function with respect to the probability measure.
\end{proof}

\begin{proof}[Proof of Theorem \ref{thm:fixedrainbowKk}]
%
%
%
For each $i \in \{1, \dots, \binom k2\}$, let $e_i$ be the unlabeled flag on $2$ vertices consisting of an edge of color $i$. Let $F$ be a complete graph on vertices $v_1, \dots, v_k$
with a fixed rainbow edge coloring. Without loss of generality, let each edge $v_i v_j$ with $i < j$ be colored with the color $(i-1)k - \binom i2 + (j-i)$. 
In other words, the edges of $F$ are colored with the set $\{1, \dots, \binom k2\}$ in the lexicographic order.
Let $\Gamma$ be an $n$-vertex graph with $C_i$ edges of color $i$ and $K$ copies of $F$.
As in the proof of Lemma \ref{lemma:blowup}, we will eventually consider the convergent sequence of blowups of $\Gamma$ whose limit is $\phi \in \text{Hom}^+(\mathcal A, \mathbb R)$. 
Thus, we may imagine that we have a large blowup graph $\Gamma_{\ell}$ on $N = \ell|V(\Gamma)|$ vertices
in which the density of edges of color $i$ satisfies $\frac{2C_i\ell^2}{N^2} \to \phi(e_i)$.
Similarly, by identifying the graph $F$ with its corresponding $k$-vertex flag, we may imagine that the density of $F$ in $\Gamma_{\ell}$ satisfies $\frac{k! K\ell^k}{N^k} \to \phi(F)$. 


We show by induction on $k$ that
\[
F \leq k! 2^{-k/2}  \left ( \prod_{i=1}^{\binom k2} e_i
\right )^{\frac{1}{k-1}}
\]
for each $k \geq 4$. 

First, for $k  = 4$, 
we write $(e_1, \dots, e_6) = \left ( \Fuu2, \Fuu3, \Fuu4, \Fuu5, \Fuu6, \Fuu7 \right )$.
An application of the Cauchy-Schwarz inequality and a $6$-color analogue of \eqref{eqn:gxb-flags} give
\begin{align*}
\Fn40{2 3 5 4 6 7}  
&=
24\cdot\left\llbracket \frac{1}{2} \Fn42{2 3 5 4 6 7} \right\rrbracket
\leq 
24\cdot\left\llbracket \Fn32{2 5 6}\times 
\Fn32{2 3 4} \right\rrbracket
\leq
24\cdot\sqrt{
\left\llbracket \Fn32{2 5 6}^2\right\rrbracket\cdot 
\left\llbracket\Fn32{2 3 4}^2\right\rrbracket
}\\
&\leq 24\cdot\sqrt{\frac{1}{4}\left(\Fn205 \times \Fn206\right)\cdot\frac{1}{4}\left(\Fn203\times\Fn204\right)}
= 6 \left(\Fn206\times\Fn205\times\Fn203\times\Fn204\right)^{1/2}.
\end{align*}
Then, by repeating the argument with an edge type $\Fll3$ of color $2$, and then again with an edge type $\Fll4$ of color $3$, we get
\begin{align}\label{eq:rainbowK4}
\Fn40{2 3 5 4 6 7} = \left(\Fn40{2 3 5 4 6 7}\right)^{1/3} \times\left( \Fn40{2 3 5 4 6 7}\right)^{1/3} \times \left(\Fn40{2 3 5 4 6 7}\right)^{1/3}
\leq
6\left( \Fn202 \times \Fn203 \times \Fn204 \times \Fn205 \times \Fn206 \times \Fn207  \right)^{1/3}.
\end{align}
This proves the claim for $k=4$.

Next, we assume that $k \geq 5$.
Let $\sigma$ be the type obtained from $F - v_1$ via the bijection $\{1, \dots, k-1\} \rightarrow  \{v_2, \dots, v_k\}$ that maps $i \mapsto v_{i+1}$, and similarly, let $F_{\sigma}$ be the $k$-vertex flag of type $\sigma$ obtained from $F$ via the injection $\{1, \dots, k-1\} \hookrightarrow \{v_2, \dots, v_k\}$ that maps $i \mapsto v_{i+1}$.
We observe that 
\[
F = k! \llbracket F_{\sigma} \rrbracket = k! \llbracket F_{\sigma} \times \sigma^{k-2} \rrbracket.
\]
Using Lemma \ref{lem:Holder} with $\ell = k-1$,
$F_1 = F_{\sigma}$ and $F_2 =  \ldots  =F_{k-1} = \sigma$,       
\begin{equation}
    \label{eqn:Holder}
F \leq k! \left ( \llbracket F_{\sigma}^{k-1} \rrbracket
\llbracket \sigma^{k-1} \rrbracket^{k-2} \right )^{1/(k-1)}
=k! \left ( \llbracket F_{\sigma}^{k-1} \rrbracket\llbracket \sigma \rrbracket^{k-2} \right )^{1/(k-1)}.
\end{equation}

We would like to represent the right-hand side of \eqref{eqn:Holder} using unlabeled flags of the type $0$.
As only one bijection $\{1, \dots, k-1\} \rightarrow V(\sigma|_0)$
gives a flag isomorphic to $\sigma$,
\begin{equation}
    \label{eqn:sigma-avg}
\llbracket \sigma \rrbracket = \frac{1}{(k-1)!} \sigma|_0.
\end{equation}
Next, we recall that $F_{\sigma}^{k-1}$ has a unique representation as a linear combination of $\sigma$-flags on $2k-2$ vertices.
Notice that $2k-2$ is the smallest possible value as $|\sigma| + (k-1) (|F_{\sigma}| - |\sigma|) = (k-1) + (k-1) = 2k-2$.
Furthermore, for each $\sigma$-flag $G_{\sigma}$ that appears in this linear combination with a non-zero coefficient,
each labeled vertex of $G_{\sigma}$ is incident to $k-1$ edges of the same color.
Given the unlabeled flag $(G_{\sigma})|_0$,
only one injection $\{1, \dots, k-1\} \hookrightarrow V(G)$
gives a labeled subgraph of $(G_{\sigma})|_0$ isomorphic to $\sigma$ in which every labeled vertex is isomorphic to $k-1$ edges of the same color; therefore,
\begin{equation}
\label{eqn:Fk-1}
\llbracket F_{\sigma}^{k-1} \rrbracket = \frac{1}{(2k-2)(2k-3) \cdots k} (F_{\sigma}^{k-1})|_0 = \frac{(k-1)!}{(2k-2)!} (F_{\sigma}^{k-1})|_0 .
\end{equation}

Next, consider an unlabeled graph $G$
that appears with a non-zero coefficient in the unique linear combination of $(2k-2)$-vertex graphs that represents
$(F_{\sigma}^{k-1})|_0$.
Let $B$ be the $K_{k-1,k-1}$ subgraph of $G$ with one partite set $B_1$ corresponding to the vertex subset $\{v_2, \dots, v_k\}$ of $V(F_{\sigma})$, and with the other partite set $B_2$ corresponding to copies of $v_1$.
Note that $B$ is the union of $k-1$ monochromatic $K_{1,k-1}$ star graphs with colors $1, \dots, k-1$.
We claim that the coefficient of $G$ in the expansion of $e_1 \times \cdots \times e_{k-1}$ is  $\frac{2^{k-1}(k-1)!}{(2k-2)!}$.
To this end, we count the number of ordered matchings $M = (a_1, \dots, a_{k-1})$ in $G$ for which each edge $a_i$ has color $i$.
As each edge of $G$ corresponding to an edge $v_i v_j$ of $F$ with $i,j \geq 2$ has a color of at least $k$,
it follows that each edge of $M$ has exactly one endpoint in $B_1$ and one endpoint in $B_2$.
Thus, 
as $B$ is a union of monochromatic copies of $K_{1,k-1}$,
we have $k-1$ choices for $a_1$, $k-2$ choices for $a_2$, and so on. Therefore, the number of valid matchings $M$ is exactly $(k-1)!$. As $G$ has exactly $\frac{(2k-2)!}{2^k}$ perfect matchings, it follows that $[G](e_1 \times \cdots \times e_{k-1}) = \frac{2^{k-1}(k-1)!}{(2k-2)!}$. Therefore,
\[
(F_{\sigma}^
{k-1})|_0 \leq \frac{(2k-2)!}{2^{k-1}(k-1)!}(e_1 \times \ldots \times e_{k-1} ).
\]
Thus, continuing from \eqref{eqn:Fk-1}, 
\begin{equation}
\label{eqn:Fk-10}
\llbracket F_{\sigma}^{k-1} \rrbracket \leq \frac{(k-1)!}{(2k-2)!} \cdot \frac{(2k-2)!}{2^{k-1}(k-1)!} (e_1 \times \ldots \times e_{k-1}) = 
2^{1-k} (e_1 \times \ldots \times e_{k-1})
\end{equation}
Combining \eqref{eqn:Holder}, \eqref{eqn:sigma-avg}, and \eqref{eqn:Fk-10},
\begin{equation}
\label{eqn:pre-induction}
F \leq \frac 12 k! \left (
\frac{1}{(k-1)!} \sigma|_0
\right )^{\frac{k-2}{k-1}} (e_1 \times \ldots \times e_{k-1})^{1/(k-1)}
\end{equation}
Finally, as $\sigma|_0$ is a rainbow edge-colored copy of $K_{k-1}$, the induction hypothesis tells us that 
\[
\sigma|_0 \leq  2^{-\frac{k-1}{2}} (k-1)! \prod_{i=k}^{\binom k2} e_i^{1/(k-2)}.
\]
Putting this together with \eqref{eqn:pre-induction},
\begin{equation}
    \label{eqn:flag-k}
F \leq \frac{k!}{2} \cdot 2^{-\frac{k-2}{2}} \prod_{i=1}^{\binom k2} e_i^{1/(k-1)} = k! 2^{-k/2} \left ( \prod_{i=1}^{\binom k2} e_i \right )^{1/(k-1)}.
\end{equation}

To finish the proof, we use Lemma~\ref{lemma:blowup}.
Let $\Gamma$ be a graph, and
let $\phi \in \text{Hom}^+(\mathcal A, \mathbb R)$ be the limit homomorphism obtained from Lemma \ref{lemma:blowup}.
If $\Gamma$ has $C_i$ edges of color $i$ for $i \in \{1,\ldots,\binom k2\}$ and $K$ copies of a fixed rainbow coloring $F$ of $K_k$, then 
\[
\phi(e_i) = \frac{2C_i}{n^2} \quad\quad\text{ and }\quad\quad \phi\left( F \right) = \frac{k!K}{n^k}.
\]
Combining this with 
 \eqref{eqn:flag-k}, we get
\begin{align*}
\frac{k! K}{n^k} &\leq k! 2^{-k/2} \left(  \prod_{i=1}^{\binom k2}  \frac{2C_i}{n^2} \right)^{1/(k-1)},
\end{align*}
which simplifies to $K^{k-1}
\leq \prod_{i=1}^{\binom k2} C_i$.

\end{proof}

\section{Proofs of Theorem~\ref {triangleent} via flag algebras, counting and entropy}

\label{sec:thm11}

In this section, we give three short proofs of Theorem \ref{triangleent}, using flag algebras, counting, and entropy.
We also prove Theorem \ref{thm:K3-stability}, which shows uniqueness of the extremal construction for Theorem \ref{thm:K3-stability}.

\begin{proof}[Flag-algebra proof of Theorem \ref{triangleent}]

By the Cauchy-Schwarz inequality,
\begin{eqnarray}
\label{eqn:tleq3sqrt}
\Fn30{4 2 3} = 6 \cdot \left\llbracket \Fllu 423  \right\rrbracket= 6 \cdot\left\llbracket \Fllu423 \times \Fll4 \right\rrbracket \leq 6\cdot \sqrt{\left\llbracket \left(\Fllu423\right)^2 \right\rrbracket} \cdot \sqrt{\left\llbracket \left(\Fll4\right)^2 \right\rrbracket }  
\leq 3 \cdot \sqrt{ \Fuu4 \times \Fuu3 \times \Fuu2}.
\end{eqnarray}
Now, let $\Gamma$ be a $3$-edge-colored graph on $n$ vertices. 
Let $\phi$ be the limit homomorphism obtained from $\Gamma$ using
Lemma~\ref{lemma:blowup}. Then, as $\phi$ is an algebra homomorphism, Lemma \ref{lemma:blowup} implies that
\begin{equation}
 \frac{6T}{n^3}  = \phi\left(\Fn30{4 2 3}\right) \leq 3\cdot \sqrt{ \phi\left(\Fuu4 \times \Fuu3 \times \Fuu2\right)} = 3\cdot\sqrt{  \frac{2R}{n^2} \cdot \frac{2G}{n^2} \cdot \frac{2B}{n^2}  },
\end{equation}
from which we conclude $T\  \leq\  \sqrt{2RGB}. $
\end{proof}

\begin{proof}[Counting proof of Theorem \ref{triangleent}]
Our argument is essentially an interpretation of \eqref{eqn:tleq3sqrt}.
Let $\Gamma$ be a graph whose edges are colored red, green, and blue, having $R$ red, $G$ green, and $B$ blue edges.
Let $\mathcal R \subseteq V(\Gamma)^2$ be the set of ordered pairs $uv \in V(\Gamma)^2$ that induce a red edge.
Note that $|\mathcal R| = 2R$.
For each $uv \in \mathcal R$, let $d^+(uv)$
denote
the number of 
triples $uvw \in V(\Gamma)^3$
for which 
$uw$ is blue and $vw$ is green. 
Then, applying the Cauchy-Schwarz inequality,
the number of rainbow triangles in $\Gamma$ is
\begin{equation}
\label{eqn:K3}
T = \sum_{uv \in \mathcal R} d^+(uv) \cdot 1 \le
\sqrt{\sum_{uv \in \mathcal R} d^+(uv)^2}\cdot 
\sqrt{2R}.
\end{equation}
As observed in the proof of Theorem \ref{thm:K4-counting}, $\sum_{uv \in \mathcal R} d^+(uv)^2$ is the number of ordered vertex tuples $(u,v,x,y)$ for which $uv$ is a red edge, $ux$ and $uy$ are blue edges, and $vx$ and $vy$ are green edges. By Lemma~\ref{lem:gxb}, $\sum_{uv \in \mathcal R} d^+(uv)^2 \leq G \cdot B$. 
Putting this together with~\eqref{eqn:K3}, we conclude 
$T \leq \sqrt{2 RGB}$.
\end{proof}

\begin{proof}[Entropy proof of Theorem~\ref{triangleent}]
    Let $\Gamma$ be a graph whose edges are colored red, green, and blue.
    Let $v_g v_b v_r\in V(\Gamma)^{3}$ be a rainbow triangle sampled uniformly at random from $\Gamma$, where $v_g v_b$, $v_b v_r$  and $v_r v_g$ are red, green and blue edges respectively. We resample a vertex $v_{r}'$ so that $v_g v_r'$ and $v_b v_r'$ are blue and green edges, and  $v_r$ and $v_r'$ are conditionally independent and identically distributed given $v_g,v_b$, so that $H(v_r'|v_g, v_b)=H(v_r|v_g, v_b)$. 

    By the chain rule, we have $H(v_g,v_b,v_r)=H(v_r|v_g,v_b)+H(v_g,v_b)$ and using $H(v_r'|v_g,v_b)=H(v_r|v_g,v_b)$, it follows $H(v_g,v_b,v_r)=H(v_r'|v_g,v_b)+H(v_g,v_b)$. By adding these two equations and using conditional independence, we obtain
    \begin{align*}
    2H(v_g,v_b,v_r)&=H(v_r|v_g, v_b)+H(v_g,v_b)+H(v_r'|v_g,v_b)+H(v_g,v_b)=H(v_r,v_r'|v_g,v_b)+2H(v_g,v_b)\\
    &=H(v_g,v_b,v_r,v_r')+H(v_g,v_b).
    \end{align*}
Since $(v_g,v_b)$ is an ordered red edge, it follows $H(v_g,v_b)\leq \log_{2}(2R)$. 
Recall that $S$ is the set of ordered vertex tuples $(u,v,x,y) \in V(\Gamma)^4$ for which $uv$ is a red edge,
$ux$ and $uy$ are blue edges, and $vx$ and $vy$ are green edges,
see Figure~\ref{fig:SS'}.
Notice that $(v_g,v_b,v_r,v_r')$ belongs 
to the set
$ S$ defined in Lemma \ref{lem:gxb}; therefore, by Lemma \ref{lem:gxb}, $H(v_g,v_b,v_r,v_r')\leq \log_{2}(|S|) \leq \log_2(|S'|) = \log_{2}(G)+\log_2(B)$. We conclude 
\[
T=2^{H(v_g,v_b,v_r)}\leq 2^{\frac 12 H(v_g,v_b,v_r,v_r')+\frac 12 H(v_g,v_b)}\leq \sqrt{2RGB}.
\] 
\end{proof}

Using our counting proof of Theorem \ref{triangleent}, 
we can in fact show uniqueness of the graph for which the upper bound on the number of rainbow triangles is attained.

\begin{proof}[Proof of Theorem~\ref{thm:K3-stability}]
     Let  $\Gamma$ be a graph whose edges are colored red, green, and blue, with exactly $T = \sqrt{2 RGB} > 0$ rainbow triangles. 
    We aim to show that $\Gamma$ is obtained from a balanced blowup of a properly colored $K_4$, by possibly adding a set
of isolated vertices. 
    Write $V' \subseteq V(\Gamma)$ for the set of vertices in $\Gamma$ incident to at least one edge.
    We also observe that each edge of $\Gamma$ belongs to a rainbow triangle. Indeed, if $e$ does not belong to a rainbow triangle, then $e$ could be deleted without reducing the number of rainbow triangles. 
    Assuming without loss of generality that $e$ is red,
    $
    T = \sqrt{2 RGB} > \sqrt{2 (R-1)GB},$
    contradicting Theorem~\ref{triangleent}.

    As Theorem \ref{triangleent} states that
    $T \leq \sqrt{2 RGB}$, the fact that $T = \sqrt{2 RGB}$ implies that each inequality in the counting proof of Theorem \ref{triangleent} is an equality.
    We show that if each inequality in the proof of Theorem \ref{triangleent} holds, then conditions  (a) and (b) of Lemma \ref{lem:K4char} hold  for $\Gamma[V']$,
    implying that $\Gamma$ is obtained from a balanced blowup of a properly colored $K_4$ by adding a set of isolated vertices.

    First, we analyze the inequality 
    $\sum_{uv \in \mathcal R} d^+(uv)^2 \leq G \cdot B$ 
    given by Lemma \ref{lem:gxb}.
    If $|S| < |S'|$ in the lemma, 
    then the inequality 
    $
\sum_{uv \in \mathcal R} d^+(uv)^2 \leq 
G \cdot B$ is strict, a contradiction; therefore, $|S| = |S'|$.
Hence, for every green edge $g$ of $G$ and every blue edge $b$ of $G$, the endpoints of $g$ and $b$ form a set $\{u,v,x,y\}$ for which $uv$ is a red edge, $ux$ and $uy$ are blue edges, and $vx$ and $vy$ are green edges.
In particular, a red edge joins an endpoint of $g$ with an endpoint of $b$.
By repeating the argument with permuted colors, we observe that the condition (a) of Lemma \ref{lem:K4char} holds for $\Gamma[V']$.

Next,  we claim that condition (b) of Lemma \ref{lem:K4char} holds for $\Gamma[V']$.
\begin{claim}
    There is a  $d \geq 1$ such that every $v \in V'$ is incident to exactly $d$ red, $d$ green, and $d$ blue edges.
\end{claim}
\begin{proof}
We analyze the Cauchy-Schwarz inequality $\sum_{uv \in \mathcal R} d^+(uv) \cdot 1 \le
\sqrt{\sum_{uv \in \mathcal R} d^+(uv)^2}\cdot 
\sqrt{2R}$.
The Cauchy-Schwarz inequality is an equality if and only if the values $d^+(uv)$ are equal for all pairs $uv \in \mathcal R$.
By repeating the argument with permuted colors, and recalling that each edge of $\Gamma$ belongs to at least one rainbow triangle, we conclude that there exists a  $d \geq 1$
such that the following holds for every permutation $(c_1,c_2,c_3)$ of the three colors: 
\begin{enumerate}
\item[$(\star)$]
If $u,v \in V(\Gamma)$ induce an edge of color $c_1$, then there are exactly $d$ vertices $w$ for which $u,w$ induce an edge of color $c_2$ and $v,w$ induce an edge of color $c_3$.
\end{enumerate}

Now,  consider a vertex $u \in V'$. As $u \in V'$, $u$ has a neighbor $u'$, and without loss of generality, $uu'$ is red. Then, by ($\star$), $u$ belongs to a rainbow triangle in which $u$ is incident to a green edge, as well as a rainbow triangle in which $u$ is incident to a blue edge. Therefore, $u$ is incident to at least one edge of each color.
    
    Next, let $(c_1,c_2,c_3)$ be an arbitrary permutation of the three colors, and let $v$ be a vertex for which $uv$ has color $c_1$. By $(\star)$, there exist at least $d$ vertices $w$ for which $uw$ has color $c_2$. Furthermore, for each vertex $w \in V(\Gamma)$ for which $uw$ has color $c_2$, condition (a) of Lemma \ref{lem:K4char}
    implies that $vw$ has color $c_3$. Therefore, $V(\Gamma)$ has exactly $d$ vertices $w$ for which $uw$ has color $c_2$. Since $u$ and $c_2$ were chosen arbitrarily, $u$ has exactly $d$ incident edges of each color.
\end{proof}
Since conditions (a) and (b) of Lemma \ref{lem:K4char} hold, $\Gamma[V']$ is a balanced blowup of a properly colored $K_4$. Therefore,  $\Gamma$ is obtained from a balanced blowup of a properly colored $K_4$ by possibly adding a set of isolated vertices.
\end{proof}

\section{Concluding Remarks}\label{sec:remarks}

We note that an argument that is similar
to but more tedious than  the proof of Theorem \ref{thm:K3-stability}
shows that 
for each $\varepsilon > 0$, there exists a  $\delta > 0$ such that when $T^2 \geq (\sqrt 2 - \delta) \sqrt{RGB}$, then $\Gamma$ can be transformed into a balanced blowup of a properly colored $K_4$ by changing at most $\varepsilon |V(\Gamma)|^2$ vertex pairs.
A precise version of \eqref{eqn:gxb-flags} states that
\begin{equation}
\label{eqn:gxb-flagsB}
\Fuu3 \times \Fuu2 =  4\cdot\left \llbracket \Fllu423 ^2 \right \rrbracket
+ \frac 16  \cdot  \Fuuuu112311 + \frac 16 \cdot  \Fuuuu112314 + 
\frac 13 \cdot  \Fuuuu212313 + \frac 12 \cdot \Fuuuu223233 +  \frac 12\cdot \Fuuuu332322 + \ldots.
\end{equation}
Thus, if $T^2 \geq  (\sqrt 2 - \delta) \sqrt{RGB}$,
then
the unlabeled graphs on the right-hand
side of \eqref{eqn:gxb-flagsB} have
density at most, say, $100\delta$.
Then, when $\delta$ is sufficiently small, an induced removal lemma (see for example~\cite{aroskar2014limitsregularityremovalfinite})
implies that $\Gamma$ can be transformed into an edge-colored graph $\Gamma'$ that contains none of the  unlabeled graphs on the right-hand side of  \eqref{eqn:gxb-flagsB}, or a color permutation thereof,
by changing at most $\frac 12 \varepsilon|V(\Gamma)|^2$ vertex pairs.
These forbidden induced subgraphs imply that $\Gamma'$ is obtained from a blowup of a properly colored $K_4$ by adding isolated vertices.
Furthermore, when $\delta$ is sufficiently small, the bound 
$T^2 \geq (\sqrt 2 - \delta) \sqrt{RGB}$
implies that $\Gamma'$ is almost balanced and hence it  can be made balanced by editing at most another $\frac 12 \varepsilon|V(\Gamma)|^2$ edges.
We decided not to present the proof, because it is tedious and not particularly enlightening.
The proof of Theorem~\ref{thm:K3-stability} uses condition (a) of Lemma \ref{lem:K4char}, which is obtained from lower-order terms and cannot be obtained in a straightforward way from the flag algebra calculation.

\section{Acknowledgments}
The authors are grateful to Bowen Li, Robert Krueger and Felix Clemen for useful discussions on the problems. In particular, Bowen Li provided a construction for Question~\ref{Q2}, which was better than the construction presented in the first version of our paper. This motivated us to involve AlphaEvolve for searching for a better construction, which is depicted in Figure~\ref{fig:Bowen}.
Experimental computer calculations were performed mainly on the Alderaan cluster.
The Alderaan cluster was funded by the NSF Campus Computing program grant OAC-2019089 and a contribution from the Simons Foundation.
It is operated by the Center for Computational Mathematics, Department of Mathematical and Statistical Sciences, College of Liberal Arts and Sciences at University of Colorado, Denver.

\bibliographystyle{plainurl}
\bibliography{references}


\ifjournalversion
\end{document}
\fi

\appendix

\section{Flag Algebras}\label{section:fa}
The purpose of this section is to introduce the main definitions for flag algebras used in this paper, 
however, we are not attempting to give a complete introduction to flag algebras.
It should allow the reader to follow the calculations and make the paper self-contained.
An interested reader may look at~\cite{fass,Lidicky2021,brosch2024gettingrootproblemsums,connor,silva2016flagalgebrasglance,andrzej}. 

To simplify the notation, denote the number of vertices of a graph $G$ by $v(G)$.
Denote all graphs on $n$ vertices up to isomorphism by $\mathcal{F}_n$ 
and the union of them by $\mathcal{F}$.
The \emph{density} of a graph $G$ in a graph $H$ is
\[
p(G,H)= \frac{|\{ X : X\subseteq V(H), H[X] \cong G \}|}{\binom{v(H)}{v(G)}}.
\]
A sequence of graphs $(G_n)_{n\geq1}$ of increasing orders is \emph{convergent} if for every $H \in \mathcal{F}$, the density of $H$ in  $(G_n)_{n\geq1}$ converges, i.e., 
$\lim_{n\to\infty} p(H,G_n)$ exists.
Examples of convergent sequences are a sequence of ba\-lanced blowups of increasing size a fixed graph or a sequence of Erd\H{o}s-R\'enyi random graphs $G_{n,p}$, where $p$ is a constant and $n$ tends to infinity. 
By compactness, every sequence of graphs has a convergent subsequence~\cite[Theorem 3.2]{RazborovFlag2007}.
Denote by $\phi=\phi (H)$ the limits of $ p(H,G_n)$, which is a function $\mathcal{F} \to [0,1]$.
Razborov showed~\cite[Theorem 3.3]{RazborovFlag2007} that $\phi$ is a homomorphism from a certain algebra $\mathcal{A}$ to $\mathbb{R}$. Positive homomorphims from $\mathcal{A}$ to $\mathbb{R}$, denoted by  $Hom^+(\mathcal{A},\mathbb{R})$, are homomorphisms $\phi$, where $\mathcal{F} \to [0,1]$. 
All these homomorphisms are  limits of convergent sequences~\cite[Theorem 3.3]{RazborovFlag2007}.  

We use graphs for accessibility of the explanation. 
The graphs can be replaced by $k$-uniform hypergraphs, permutations, or other models. In general one needs to have vertices with relations of finite arity.
Different models result in different algebras $\mathcal{A}$.

If an expression using flags is valid for all $\phi \in Hom^+(\mathcal{A},\mathbb{R})$, then we omit writing $\phi$ to decrease clutter in the notation.

The algebra $\mathcal{A}$ is obtained from formal finite linear combinations of graphs in $\mathcal{F}$ by factoring $\mathcal{K} := \mathrm{span}(\{ F - \sum_{F'\in \mathcal{F}_{v(F)+1}}  p(F,F')F) \,:\, \forall F \in \mathcal{F}\}$. The expressions in $\mathcal{K}$ enforce in $\mathcal{A}$ identities such as
\[
\Fn202 = \frac{1}{3}\Fn30{1 1 2} + \frac{2}{3}\Fn30{1 2 2} + \Fn30{2 2 2}.
\]
The intuitive idea is that calculations happen with linear combinations of densities of small graphs in a very large graph. 
The product of $F_1,F_2 \in \mathcal{F}$ is defined as
\begin{align}
F_1 \times F_2 = \sum_{F \in \mathcal{F}_{v(F_1)+v(F_2)}} p(F_1,F_2; F)\cdot F,\label{eq:times}
\end{align}
where $p(F_1,F_2; F)$ is the probability that $F[X]\cong F_1$ and $F[V(F)\setminus X]\cong F_2$ for $X \subseteq V(F)$ with $|X| = v(F_1)$  which is chosen uniformly at random.
One could think of the left hand-side of \eqref{eq:times} as choosing uniformly independently at random in a large graph sets $X_1$ of $v(F_1)$ vertices and $X_2$ of $v(F_2)$ vertices and asking if $X_1$ induces a copy of $F_1$ and $X_2$ induces a copy of $F_2$. If the underlying graph is very large, then $X_1$ and $X_2$ are typically disjoint. The right hand-side of \eqref{eq:times} lists the options for $X_1 \cup X_2$. 
For $\mathcal{F}$ of simple graphs
\[
\Fuu2 \times \Fuu1 = \frac{1}{6}\left(
\Fn{4}{0}{1 1 2 1 1 1}+
2\Fn{4}{0}{1 1 2 1 1 2}+
3\Fn{4}{0}{1 1 2 1 2 2}+
3\Fn{4}{0}{1 2 2 1 1 2}+
\Fn{4}{0}{2 1 2 1 1 2}+
2\Fn{4}{0}{2 2 2 1 1 2}+
\Fn{4}{0}{2 2 2 1 2 2}
\right).
\]
For $\mathcal{A}$ of $3$-edge colored graphs
\[
\Fuu2 \times \Fuu3
=
\frac{1}{6}
\left(
\Fn{4}{0}{1 1 2 3 1 1}+
\Fn{4}{0}{1 1 2 3 1 2}+
\Fn{4}{0}{1 1 2 3 1 3}+
2\Fn{4}{0}{1 2 2 3 3 1}+
2\Fn{4}{0}{1 3 2 3 2 2}+
2\Fn{4}{0}{1 3 2 3 2 3}+
2\Fn{4}{0}{1 3 2 3 2 4}+
3\Fn{4}{0}{2 3 2 3 2 3}+
3\Fn{4}{0}{2 2 2 3 3 3}+
\ldots
\right).
\]
Extending \eqref{eq:times} linearly defines the product on $\mathcal{A}$.

Extremal graph theory arguments often include counting over a fixed substructure. Some examples for such counting are $\sum_{v\in V(G)}d(v)$ or $\sum_{uv\in E(G)} |N(u) \cap N(v)|$, where $G$ is a graph.
The entries in these sums have some fixed distinguished vertices. 
In flag algebras, this is modeled using graphs with $\ell \geq 0$ vertices labeled by $\{1,2,\ldots,\ell\}$. 
The graph on $\ell$ vertices, where each of the $\ell$ vertices are  labeled is called a \emph{type}. A type is usually denoted by $\sigma$ and the resulting algebra by $\mathcal{A}^\sigma$.
The main new properties of them are that the isomorphisms must preserve the labeling of the labeled vertices  
and in the definition of
product, the labeled vertices are shared.
More formally, let $(F_1,\theta_1)$ and $(F_2,\theta_2)$
be two labeled flags of the same type $\sigma$, where $F_i$ is an unlabeled graph and $\theta_i : [\ell] \hookrightarrow V[F_i]$ is an injective map indicating the labeled vertices for $i \in \{1,2\}$. 
Recall that $(F_i,\theta_i)$ being of type $\sigma$ means that $(F_i[\text{Im}(\theta_i)],\theta_i)$ is isomorphic to $\sigma$.
The product is defined as
\[
(F_1,\theta_1)\times (F_2,\theta_2)=\sum_{(F,\theta)\in\mathcal{F}^\sigma_{v(F_1)+v(F_2)-\ell}} p((F_1,\theta_1),(F_2,\theta_2);(F,\theta)) \cdot (F,\theta),
\]
where $p((F_1,\theta_1),(F_2,\theta_2);(F,\theta))$ is the probability that the set $X \subseteq V(F)\setminus \text{Im}(\theta)$ with $|X|=v(F_1)-\ell$ sampled uniformly at random satisfies that
$(F[X\cup \text{Im}(\theta)], \theta)$ is isomorphic to $(F_1,\theta_1)$
and $(F[ V(F)\setminus X], \theta)$ is  isomorphic to $(F_2,\theta_2)$.
In $3$-edge colored 
graphs\\
\[
\left(\Fn{3}{2}{2 3 4}\right)^2 = 
\Fn{4}{2}{2 3 3 4 4 1}+
\Fn{4}{2}{2 3 3 4 4 2}+
\Fn{4}{2}{2 3 3 4 4 3}+
\Fn{4}{2}{2 3 3 4 4 4}
\]
and
\[
\Fn{3}{2}{2 3 4} \times \Fn{3}{2}{2 4 3}  = 
\frac{1}{2}\Fn{4}{2}{2 3 4 4 3 1}+
\frac{1}{2}\Fn{4}{2}{2 3 4 4 3 2}+
\frac{1}{2}\Fn{4}{2}{2 3 4 4 3 3}+
\frac{1}{2}\Fn{4}{2}{2 3 4 4 3 4}.
\]
Expressions
$\frac{1}{v}\sum_{v\in V(G)}d(v)$ or $\frac{1}{|E(G)|}\sum_{uv\in E(G)} |N(u) \cap N(v)|$ for a graph $G$ are analogous to a linear unlabeling operator in flag algebras $\llbracket \cdot \rrbracket$.
Let $(F,\theta)$ be a labeled graph, where $F$ is an unlabeled graph and an injective map $\theta:[\ell] \to V(F)$ gives the labeling. Then,
\[
\llbracket (F,\theta) \rrbracket = c_F F,
\],
where $c_F$ is the probability that $(F,\theta) \cong (F,\theta')$ where
$\theta':[\ell] \to V(F)$ is an injective map chosen uniformly at random.
In other words, randomly label $\ell$ vertices of $F$ and ask if  the original labeled graph is obtained.
For example
\[
\left \llbracket \Flluu423324 \right \rrbracket = \frac 13 \Fuuuu423324 \quad\text{,}\quad  \left\llbracket \Fll2 \right\rrbracket = \Fuu2
\quad\text{and}\quad 
\left\llbracket\Flluu422331\right\rrbracket = \frac{1}{12} \Fuuuu422331.
\]
Notice that  the multiplication is defined only for graphs, where the subgraphs induced by the labeled vertices, i.e., the types, are the same.

With these definitions, there is an analogue of the Cauchy-Schwarz inequality \cite[Theorem 3.14]{RazborovFlag2007} that states for $f,g \in \mathcal{A}^{\sigma}$
\[
\llbracket f^2\rrbracket \cdot \llbracket g^2\rrbracket \geq  \llbracket fg\rrbracket^2.
\]
All calculations presented in this paper should formally be surrounded by $\phi( \cdots )$ that is a limit of a convergent sequence of graphs. 
The calculations would be on graph densities.
Since wrapping the calculations in $\phi( \cdot )$ adds notation but it is not useful for the calculation itself, they are usually  omitted. 
The calculations are valid for any choice of $\phi$ and can be intuitively thought of as calculations in large graphs.
Notice that the results from flag algebras are valid only in limits. 
When interpreting the calculations in large graphs, they come with $o(1)$ additive error. 
In particular, the calculations do not directly imply any meaningful thing if the convergent sequence consists of sparse graphs. In the flag algebras world, a sequence of sparse graphs is indistinguishable from graphs with no edges at all.\\

\end{document}